\documentclass[11pt]{amsart}
\usepackage{a4wide,amssymb}

\newcommand{\V}{\mathcal V}
\newcommand{\U}{\mathcal U}
\newcommand{\C}{\mathcal C}
\newcommand{\E}{\mathcal E}
\newcommand{\I}{\mathcal I}
\newcommand{\J}{\mathcal J}
\newcommand{\A}{\mathcal A}
\newcommand{\IN}{\mathbb N}
\newcommand{\cof}{\mathrm{cof}}
\newcommand{\add}{\mathrm{add}}

\newcommand{\e}{\varepsilon}
\newcommand{\Ra}{\Rightarrow}
\newcommand{\w}{\omega}
\newcommand{\IR}{\mathbb R}

\newtheorem{theorem}{Theorem}[section]
\newtheorem{corollary}[theorem]{Corollary}
\newtheorem{lemma}[theorem]{Lemma}
\newtheorem{problem}[theorem]{Problem}
\newtheorem{proposition}[theorem]{Proposition}
\theoremstyle{definition}
\newtheorem{definition}[theorem]{Definition}
\newtheorem{example}[theorem]{Example}

\title{The normality and bounded growth of balleans}
\author{Taras Banakh and Igor Protasov}
\address{T.Banakh: Ivan Franko National University of Lviv (Ukraine) and Jan Kochanowski University in Kielce (Poland)}
\email{t.o.banakh@gmail.com}
\address{I.Protasov: Taras Shevchenko National University in Kyiv, Ukraine}
\email{i.v.protasov@gmail.com} 
\subjclass{54D15, 20F65}
\keywords{Ballean, coarse space, normality, bounded growth,  product of balleans, hyperballean, symmetric power, group, $G$-space.}

\begin{document}
\begin{abstract} By a ballean we understand a set $X$ endowed with a family of entourages which is a base of some coarse structure on $X$. Given two unbounded balleans $X,Y$ with normal product $X\times Y$, we prove that the balleans $X,Y$ have bounded growth and the bornology of $X\times Y$ has a linearly ordered base. A ballean $(X,\E_X)$ is defined to have {\em bounded growth} if there exists a function $G$ assigning to each point $x\in X$ a bounded subset $G[x]\subset X$ so that for any bounded set $B\subset X$ the union $\bigcup_{x\in B}G[x]$ is bounded and for any entourage $E\in\E_X$ there exists a bounded set $B\subset X$ such that $E[x]\subset G[x]$ for all $x\in X\setminus B$. We prove that the product $X\times Y$ of two balleans has bounded growth if and only if $X$ and $Y$ have bounded growth and the bornology of the product $X\times Y$ has a linearly ordered base. Also we prove that a ballean $X$ has bounded growth (and the bornology of $X$ has a linearly ordered base) if its symmetric square $[X]^{\le 2}$ is normal (and the ballean $X$ is not ultranormal).
A ballean $X$ has bounded growth and its bornology has a linearly ordered base if  for some $n\ge 3$ and some subgroup $G\subset S_n$ the $G$-symmetric $n$-th power $[X]^n_G$ of $X$ is normal. On the other hand, we prove that for any ultranormal discrete ballean $X$ and every $n\ge 2$ the power $X^n$ is not normal but the  hypersymmetric power $[X]^{\le n}$ of $X$ is normal. Also we prove that the finitary ballean of a group is normal if and only if it has bounded growth if and only if the group is countable.
\end{abstract}

\maketitle 

\section{Introduction and survey of results}

It is well-known that the normality of topological spaces is not preserved by products. For example, the Sorgenfrey line is normal but is square does not.

In this paper we study the normality of products of balleans. For this purpose we introduce a new notion, called the bounded growth. We start with necessary definitions.

A {\em ballean} is a pair $(X,\mathcal E_X)$ consisting of a set $X$ and a family $\mathcal E_X$ of subsets of the square $X\times X$  satisfying the following three axioms:
\begin{enumerate}
\item each $E\in\mathcal E_X$ contains the diagonal $\Delta_X=\{(x,x):x\in X\}$ of $X$;
\item for any $E,F\in\mathcal E_X$ there exists $D\in\mathcal E_X$ such that $E\circ F^{-1}\subset D$, where  $E\circ F:=\{(x,z):\exists y\in X$ such that $(x,y)\in E$ and $(y,z)\in F\}$ and $F^{-1}:=\{(y,x): (x,y)\in F\}$.
\item $\bigcup\E_X=X\times X$.
\end{enumerate}
The family $\mathcal E_X$ is called the {\em ball structure} of the ballean $(X,\mathcal E_X)$ and its elements are called {\em entourages}. For each entourage $E\in\mathcal E_X$ and point $x\in X$ we can consider the set $E[x]:=\{y\in X:(x,y)\in E_X\}$ called the {\em ball of radius $E$} centered at $x$. 
For a subset $A\subset X$ the set $E[A]:=\bigcup_{a\in A}E[x]$ is called the {\em $E$-neighborhood} of  $A$. Observe that $E=\bigcup_{x\in X}\{x\}\times E[x]$, so the entourage $E$ can be recovered from the family of balls $E[x]$, $x\in X$. 


For a ballean $(X,\E_X)$ and a subset $Y\subset X$ the ballean $(Y,\E_X{\restriction}Y)$ endowed with the ball structure
$$\E_X{\restriction}Y:=\{(Y\times Y)\cap E:E\in\E_X\}$$is called a {\em subballean} of $X$.

Any metric space $(X,d)$ carries a natural ball structure $\{E_\e:0<\e<\infty\}$ consisting of the entourages $E_\e:=\{(x,y)\in X\times X:d(x,y)<\e\}$.

A ballean $(X,\mathcal E)$ is called a {\em coarse space} if for any entourage $E\in\mathcal E_X$, any set $F\subset E$ with $\Delta_X\subset F$ belongs to $\mathcal E_X$.  In this case $\mathcal E_X$ is called the {\em coarse structure} of $X$. For a coarse structure $\E$, a subfamily  $\mathcal B\subset\E$ is called a {\em base} of $\mathcal E$ if each set $E\in\E$ is contained in some set $B\in \mathcal B$. It is easy to see that each base of a coarse structure is a ball structure. On the  other hand, each ball structure $\mathcal E$ on a set $X$ is a 
base of the unique coarse structure $${\downarrow}\mathcal E:=\{E\subset X\times X:\Delta_X\subset E\subset F\mbox{ for some }F\in\mathcal B\}.$$ 

If the ball (or coarse) structure $\E_X$ is clear from the context, we shall write $X$ instead of $(X,\E_X)$.
 
For a family $(X_i)_{i\in I}$ of balleans their product $\prod_{i\in I}X_i$ carries the natural ball structure $$\Big\{\big\{\big((x_i)_{i\in I},(y_i)_{i\in I}\big):\forall i\in I\;(x_i,y_i)\in E_i\big\}:(E_i)_{i\in I}\in\prod_{i\in I}\E_{X_i}\Big\}.$$
The ballean $\prod_{i\in I}X_i$ will be called the {\em box-product} of the balleans $X_i$, $i\in I$. 
If the index set $I$ is finite, then the box-product $\prod_{i\in I}X_i$ will be called the {\em product} of balleans. If $X_i=X$ for all $i\in I$ then the product $\prod_{i\in I}X_i$ is denoted by $X^I$ and is called the {\em $I$-th power} of $X$. 
\smallskip

A subset $B\subset X$ of a ballean $(X,\E_X)$ is called {\em bounded} if $B\subset E[x]$ for some $E\in\E_X$ and $x\in X$. A ballean $X$ is {\em bounded} if $X$ is a bounded set in $(X,\E_X)$. 

The family $\mathcal B_X$ of all bounded subsets is called the {\em bornology} of the ballean $(X,\E_X)$. If the ballean $X$ is unbounded, then the bornology $\mathcal B_X$ is an ideal of subsets of $X$. A family $\I$ of subsets of a set $X$ is called an {\em ideal} on $X$ if $\I$ is closed under finite unions and taking subsets, and $X\notin\I$. 

More information on balleans and coarse spaces can be found in the monographs  \cite{CH}, \cite{NYu},  \cite{PB}, \cite{PZ}, \cite{Roe} and in the papers \cite{DZ1}, \cite{DZ2}, \cite{Dr}, \cite{DH}, \cite{Nek}. 

Now we recall the necessary information on normal balleans (which were introduced and studied by Protasov in \cite{Prot}).

Let $(X,\mathcal E_X)$ be a ballean. Two subsets $A,B\subset X$ are called {\em asymptotically disjoint} if for any $E\in\E_X$ the intersection $E[A]\cap E[B]$ is bounded in $(X,\mathcal E_X)$. We recall that $E[A]=\bigcup_{a\in A}E[a]$ is the $E$-neighborhood of $A$ in $(X,\E_X)$.

A subset $U\subset X$ is called an {\em asymptotic neighborhood} of a set $A\subset X$ if for every $E\in\E_X$ the set $E[A]\setminus U$ is bounded. It is easy to see that a subset $U\subset X$ is an asymptotic neighborhood of a set $A\subset X$ if and only if the sets $A$ and $X\setminus U$ are asymptotically disjoint.

A ballean $(X,\E)$ is called {\em normal} if any asymptotically disjoint sets $A,B\subset X$ have disjoint asymptotic neighborhoods. A ballean $(X,\E_X)$ is called {\em ultranormal} of $X$ contains no  asymptotically disjoint unbounded sets $A,B\subset X$. It is clear that each ultranormal ballean is normal. Examples of ultranormal balleans will be presented in Examples~\ref{e:ultraideal} and \ref{e:final}.

Normal balleans have properties, analogous to properties of normal topological spaces. For example, Protasov \cite{Prot} proved analogs of Urysohn Lemma and Titze-Urysohn Extension Theorem for normal balleans.  By Proposition 1.2 in \cite{Prot}, the normality of balleans is inherited by subballeans. In Section~\ref{s:normpres} we prove more results on preservation of normality by maps between balleans.


A sufficient condition for normality is given in the following theorem proved by Protasov \cite{Prot}.

\begin{theorem}\label{t:prot} A ballean $X$ is normal if its coarse structure has a linearly ordered base.
\end{theorem}

Theorem~\ref{t:prot} motivates the problem of recognizing balleans whose coarse structure has a linearly ordered base. We shall answer this problem using two  cardinal characteristics of balleans: the {\em additivity} $\add(\E_X)$ and the {\em cofinality} $\cof(\E_X)$ of the ball structure $\E_X$.

Those cardinal characteristics are defined for each partially ordered set $(P,\le)$ as follows:
$$
\begin{aligned}
\cof(P)&:=\min\{|A|:A\subset P\;\;\forall x\in P\;\exists y\in A\;\;\;(x\le y)\},\\
\add(P)&:=\min\{|A|:A\subset P\;\;\forall y\in P\;\exists x\in A\;\;\;(x\not\le y)\}.
\end{aligned}
$$
The cardinal $\add(P)$ is not defined if the partially ordered set $P$ has the largest element. In this case we put $\add(P)=\cof(P)=1$.

Observe that $\cof(P)$ is the smallest cardinality of a cofinal set in $P$ and $\add(P)$ is the smallest cardinality of an unbounded set in $P$, where a subset $A\subset P$ is
\begin{itemize}
\item {\em bounded} if there exists $y\in P$ such that $x\le y$ for all $x\in A$;
\item {\em cofinal} if for each $x\in P$ there exists $y\in A$ such that $x\le y$.
\end{itemize}
Cofinal sets are also called {\em bases} of partially ordered sets.
It is easy to see that $\add(P)\le\cof(P)$ for any partially ordered set $P$. Moreover, $P$ has a linearly ordered base if and only if $\add(P)=\cof(P)$ if and only if $P$ has a well-ordered base of cardinality $\add(P)=\cof(P)$. An important observation is that for any partially ordered set $P$ the cardinal $\add(P)$ is regular.

Each ballean $X$ supports three natural structures, which are partially ordered sets: the ball structure $\E_X$, the coarse structure ${\downarrow}\E_X$, and the bornology $\mathcal B_X$. The cardinal characteristics of these partially ordered sets can be considered as cardinal characteristics of the ballean.

Let us observe that $$\add(\E_X)=\add({\downarrow}\E_X)\le\add(\mathcal B_X)\le \cof(\mathcal B_X)\le\cof({\downarrow}\E_X)=\cof(\E_X)$$ for every ballean $X$.


It follows that the ball structure $\E_X$ of a ballean $X$ has a linearly ordered base if and only if $\add(\E_X)=\cof(\E_X)$ if and only if $X$ has a well-ordered base of cardinality $\add(\E_X)=\cof(\E_X)$. 

A ballean $X$ is defined to be {\em $\cof$-regular} if $\cof(\E_X)=\cof(\mathcal B_X).$ 
Many natural examples of balleans are $\cof$-regular. In particular, so are balleans described in Examples 1.1, 1.2 in \cite{Prot} and Examples 2,3,6 in \cite{P2008}.

Theorem~\ref{t:prot} implies that a ballean $X$ is normal if $\cof(\E_X)\le\add(\E_X)$. In fact, the normality of $X$ can be derived from the weaker inequalities $\cof(\E_X)\le\add(\mathcal B_X)$,  $\cof_*(\E_X)\le\add(\mathcal B_X)$ or even $\cof_{\star}(\E_X)\le\add(\mathcal B_X)$.

Here $\cof_*(\E_X)$ is defined as the smallest cardinality $|\C|$ of a subfamily $\C\subset\E_X$ such that for any entourage $E\in\E_X$ there exists an entourage $C\in C$ such that $E\setminus C$ is bounded in $X\times X$.

The cardinal characteristic $\cof_\star(\E_X)$ is defined as 
$$\cof_\star(\E_X):=\sup_{A\subset X}\cof_*(\E_X[A])$$where $\cof_*(\E_X[A])$ is the smallest cardinality $|\C|$ of a subfamily $\C\subset \E_X$ such that for any $E\in\E_X$ there exists $C\in\C$ such that $E[A]\setminus C[A]\in\mathcal B_X$. It is easy to see that $$\cof_\star(\E_X)\le\cof_*(\E_X)\le\cof(\E_X)=\max\{\cof_*(\E_X),\cof(\mathcal B_X)\}$$ for every ballean $(X,\E_X)$.

\begin{theorem}\label{t:prot2} A ballean $X$ is normal if $\cof_\star(\E_X)\le\add(\mathcal B_X)$.
\end{theorem}

This theorem will be proved in Section~\ref{s:prot2}. Since $\cof_\star(\E_X)\le\cof(\E_X)$, it implies the following sufficient condition of normality in box-products. 

\begin{corollary}\label{c1} The box-product $\prod_{i\in I}X_i$ of balleans is normal if $$\cof(\mathcal E_{X_i})=\add(\mathcal B_{X_j})>|I|\mbox{ for any $i,j\in\I$.}$$In particular, a ballean $X$ is normal if $\cof(\E_X)\le\add(\mathcal B_X)$.
\end{corollary}

The necessary conditions of the normality of ballean products are given in the following theorem that will be proved in Section~\ref{s:proofs}.

\begin{theorem}\label{t:main} If the product $X\times Y$ of two unbounded  balleans is normal, then
\begin{enumerate}
\item $\add(\mathcal B_{X})=\cof(\mathcal B_{X})=\add(\mathcal B_Y)=\cof(\mathcal B_Y)$;
\item the bornology $\mathcal B_{X\times Y}$ of the product $X\times Y$ has a linearly ordered base;
\item the balleans $X,Y$ have bounded growth.
\end{enumerate}
\end{theorem}

The notion of bounded growth that appears in the preceding theorem is new and is defined as follows.

\begin{definition}\label{d:bg} A ballean $(X,\E_X)$ is defined to have {\em bounded growth} if there exists a subset $G\subset X\times X$ such that
\begin{itemize}
\item for every bounded set $B\subset X$ the set $G[B]=\{y\in X:\exists x\in B\;(x,y)\in G\}$ is  bounded in $X$;
\item for every entourage $E\in\E_X$ there exists a bounded set $B\subset X$ such that $E[x]\subset G[x]$ for all $x\in X\setminus B$.
\end{itemize} 
The function $G$ will be refered to as a {\em growth entourage} of the ballean $X$.
\end{definition}

In Section~\ref{s:growth} we study balleans of bounded growth and prove the following characterization.

\begin{theorem}\label{t:main2} The product $X\times Y$ of two balleans $X,Y$ has  bounded growth if and only if the balleans $X,Y$ have bounded growth and the bornology $\mathcal B_{X\times Y}$ has a linearly ordered base.
\end{theorem}

Let us observe that Theorem~\ref{t:main} implies the following corollary that  nicely complements Corollary~\ref{c1}.

\begin{corollary}\label{c2} If the box-product $\prod_{i\in I}X_i$ of $|I|>1$ unbounded balleans is normal, then each ballean $X_i$ has bounded growth and $\add(\mathcal B_{X_i})=\cof(\mathcal B_{X_j})$ for any $i,j\in I$.
\end{corollary}

Theorems~\ref{t:prot} and \ref{t:main} will be applied to prove the following characterization of normality of finite products of $\cof$-regular balleans. Let us recall that a ballean $X$ is {\em $\cof$-regular} if $\cof(\mathcal E_X)=\cof(\mathcal B_X)$. 

\begin{theorem}\label{t:bornoreg} The finite product $X=\prod_{i=1}^nX_i$ of $n\ge 2$ unbounded \textup{(}$\cof$-regular\textup{)} balleans is normal \textup{(}if and\textup{)} only if the bornology $\mathcal B_X$ of $X$ has a linearly ordered base.
\end{theorem}

\begin{proof} Assuming that the product $X=\prod_{i=1}^n X_i$ is normal, we conclude that for any $1\le i<j\le n$ the product $X_i\times X_j$ is normal, too. Then $\add(\mathcal B_{X_i})=\add(\mathcal B_{X_j})=\cof(\mathcal B_{X_i})=\cof(\mathcal B_{X_j})$ according to Theorem~\ref{t:main}. So, there exists an infinite regular cardinal $\kappa$ such that $\add(\mathcal B_{X_i})=\cof(\mathcal B_{X_i})=\kappa$ for all $i\le n$. It follows that for every $i\le n$ the bornology $\mathcal B_{X_i}$ has a base $(B_{i,\alpha})_{\alpha\in\kappa}$ such that $B_{i,\alpha}\subset B_{i,\beta}$ for all $\alpha<\beta<\kappa$. Then $\big\{\prod_{i=1}^k B_{i,\alpha}\}_{\alpha\in\kappa}$ is a well-ordered base of the bornology $\mathcal B_X$.
\smallskip

Now assume that the bornology $\mathcal B_X$ of $X$ has a linearly ordered base and the balleans $X_i$ are $\cof$-regular. Then  
$$\cof(\E_X)=\max_{1\le i\le n}\cof(\E_{X_i})=\max_{i\le i\le n}\cof(\mathcal B_{X_i})=\cof(\mathcal B_X)=\add(\mathcal B_X).$$ So we can apply Corollary~\ref{c1} and conclude that the ballean $X$ is normal.
\end{proof}

Important examples of  $\cof$-regular balleans of bounded growth are discrete balleans.

\begin{definition}\label{d:dis} A ballean $(X,\E_X)$ is called {\em discrete} if $X$ is unbounded and for any entourage $E\in\E_X$ there exists a bounded set $B_E\subset X$ such that $E[x]=\{x\}$ for all $x\in X\setminus B_E$.
\end{definition}

Discrete balleans are called pseudodiscrete in \cite{PZ} and thin in \cite{LP}. 
 The coarse structure ${\downarrow}\E_X$ of a discrete ballean $(X,\E_X)$ can be recovered from the bornology $\mathcal B_X$: an entourage $E\subset X\times X$ belongs to ${\downarrow}\E_X$ if and only if $E\subset \Delta_X\cup(B\times B)$ for some bounded set $B\in\mathcal B_X$. This observation implies that  $\add(\E_X)=\add(\mathcal B_X)\mbox{  and  }\cof(\E_X)=\cof(\mathcal B_X)$, so each discrete ballean is $\cof$-regular. Definitions~\ref{d:bg} and \ref{d:dis} imply that discrete balleans have bounded growth. Each discrete ballean $(X,\E_X)$ has $\cof_\star(\E_X)=\cof_*(\E_X)=|\{\Delta_X\}|=1$ and hence is normal by Theorem~\ref{t:prot2}. A ballean is called {\em ultradiscrete} if it is ultranormal and discrete.
 
\begin{example}\label{e:ultraideal} Each ideal $\mathcal B$ of subsets of a set $X=\bigcup\mathcal B$ induces the discrete coarse structure ${\Downarrow}\mathcal B$, generated by the base consisting of the entourages $(B\times B)\cup\Delta_X$ where $B\in\mathcal B$. The bornology $\mathcal B_X$ of the ballean $(X,{\Downarrow}\mathcal B)$ coincides with the ideal $\mathcal B$. Being discrete, the ballean $(X,{\Downarrow}\mathcal B)$ is $\cof$-regular and has bounded growth. It is ultranormal if and only if $\mathcal B$ is a maximal ideal on $X$ if and only if the family $\{X\setminus B:B\in\mathcal B\}$ is an ultrafilter.
\end{example}
 
In Proposition~\ref{p:ultra} we shall prove that any ultradiscrete ballean $X$ has 
$\add(\mathcal B_X)<\cof(\mathcal B_X)$, which implies that the bornology of $X$ does not have a linearly ordered base. Combining this fact with Theorems~\ref{t:bornoreg} and \ref{t:prot2}, we obtain

\begin{corollary} For any ultradiscrete balleans $X,Y$ the product $X\times Y$ is not normal and hence $$\cof_*(\mathcal E_{X\times Y})\ge\cof_\star(\E_{X\times Y})>\add(\mathcal B_{X\times Y})\ge 1=\cof_*(\E_X)=\cof_*(\E_Y).$$
\end{corollary}

In particular, the square $X\times X$ of an ultradiscrete ballean $X$ is not normal. Surprisingly, but the hypersymmetric powers of ultradiscrete balleans are normal.

For a ballean $X$ by $[X]^{\mathcal B}$ we denote the family $\mathcal B_X\setminus\{\emptyset\}$ of all non-empty bounded sets in $X$, endowed with the ball structure $ \E_{[X]^{\mathcal B}}$ consisting of the entourages $$\hat E=\{(A,B)\in[X]^{\mathcal B}\times[X]^{\mathcal B}:A\subset E[B],\;\;B\subset E[A]\}$$where $E\in\E_X$. The ballean $([X]^{\mathcal B},\E_{[X]^{\mathcal B}})$ is called the {\em hyperballean} of  $(X,\E_X)$. Hyperballeans were studied in \cite{ProtEJM} and \cite{DPPZ}.
\smallskip

For a natural number $n$ the subbalean 
$$[X]^{\le n}:=\{A\in[X]^{\mathcal B}:|A|\le n\}$$ of $[X]^{\mathcal B}$ is called the {\em hypersymmetric $n$-th power} of the ballean $X$. It is clear that $X$ can be identified with the ballean $[X]^{\le 1}$.

It is easy to see that the cardinal characteristics $\cof(\E_X)$, $\add(\E_X)$, $\cof(\mathcal B_X)$ and $\add(\mathcal B_X)$ of a ballean $X$ coincide with the corresponding cardinal characteristics of its hyperballean. So, we can apply Corollary~\ref{c1} and obtain the following proposition.

\begin{proposition}\label{p:lin-hyp} The hyperballean $[X]^{\mathcal B}$ of a ballean $X$ is normal, if $\cof(\E_X)\le\add(\mathcal B_X)$.
\end{proposition}

In spite of the fact, that for an ultradiscrete ballean $X$ the square $X\times X$ is not normal, in Section~\ref{s:hyper1} we shall prove the following inexpected result.

\begin{theorem}\label{t:ultranorm} For any ultradiscrete ballean $X$ and every $n\ge 2$ the $n$-th power $X^n$ of $X$ is not normal but the hypersymmetric power $[X]^{\le n}$ of $X$ is normal.
\end{theorem}

On the other hand, in Section~\ref{s:hyper2} we shall prove the following necessary conditions of the normality of the hypersymmetric square.

\begin{theorem}\label{t:hyper2} If for a ballean $X$ the hypersymmetric square $[X]^{\le 2}$ is normal, then 
\begin{enumerate}
\item $X$ has bounded growth and 
\item either $X$ is ultranormal or the bornology $\mathcal B_X$ of $X$ has a linearly ordered base.
\end{enumerate}
\end{theorem}

Theorems~\ref{t:bornoreg}, \ref{t:ultranorm}, \ref{t:hyper2} and Proposition~\ref{p:lin-hyp} imply the following characterization.

\begin{corollary} For a discrete ballean $X$ the following conditions are equivalent:
\begin{enumerate}
\item for every $n\in\IN$ the hypersymmetric power $[X]^{\le n}$ of $X$ is normal;
\item the hypersymmetric square $[X]^{\le 2}$ is normal;
\item either $X$ is ultranormal or the bornology $\mathcal B_X$ of $X$ has a linearly ordered base;
\item $X$ is ultranormal or $X\times X$ is normal.
\end{enumerate}
\end{corollary}

\begin{proof} The implication $(1)\Ra(2)$ is trivial, $(2)\Ra(3)$ 
follows from Theorem~\ref{t:hyper2} and $(3)\Ra(1)$ follows Theorem~\ref{t:ultranorm} and Proposition~\ref{p:lin-hyp}. The equivalence $(3)\Leftrightarrow(4)$ follows from Theorem~\ref{t:bornoreg} and the $\cof$-regularity of discrete balleans.
\end{proof}

Also we shall consider the construction of the $G$-symmetric power $[X]^n_G$, which is intermediate between the constructions $X^n$ and $[X]^{\le n}$ of $n$-th power and hypersymmetric $n$-th power. Given a subgroup $G\subset S_n$ of the permutation group $S_n$ of a natural number $n:=\{0,\dots,n-1\}$, for every ballean $X$ consider the quotient space $[X]^n_G$ of $X^n$ by the equivalence relation $\sim_G$ defined by $x\sim_G y$ for $x,y\in X^n$ iff $y=x\circ g$ for some permutation $g\in G$.
For any $x\in X^n$ (which is a function $x:n\to X$) by $xG:=\{x\circ g:g\in G\}\in[X]^n_G$ we shall denote its $\sim_G$-equivalence class. The set $[X]^n_G:=\{xG:x\in X^n\}$ is endowed with the ball structure $\E_{[X]^n_G}$ consisting of the entourages
$$\hat E:=\{(xG,yG):(x,y)\in E\}$$where $E\in\E_{X^n}$. The ballean $([X]^n_G,\E_{[X]^n_G})$ is called {\em the $G$-symmetric $n$-th power} of $X$. If the group $G$ is trivial, then $[X]^n_G=X^n$. So, the construction of a $G$-symmetric  
$n$-th power $[X]^n_G$ generalizes the construction of the $n$-th power $X^n$ of a ballean $X$. The ballean $[X]^n_{S_n}$ will be denoted by $[X]^n$ and called the {\em symmetric $n$-th power} of $X$. For $n=2$ the {\em symmetric square} $[X]^2$ can be  identified with the hypersymmetric square $[X]^{\le 2}$ of $X$.  

It is easy to see that the cardinal characteristics $\cof(\E_X)$, $\add(\E_X)$, $\cof(\mathcal B_X)$ and $\add(\mathcal B_X)$ of a ballean $X$ coincide with the corresponding cardinal characteristics of its $G$-symmetric powers $[X]^n_G$. So, we can apply Corollary~\ref{c1} and obtain the following proposition.

\begin{proposition}\label{p:Gnorm} If a ballean $X$ has $\cof(\E_X)\le\add(\mathcal B_X)$, then for every $n\in\IN$ and every subgroup  $G\subset S_n$ the ballean $[X]^n_G$ is normal.
\end{proposition}

The interplay between the normality of $G$-symmetric powers for various groups $G\subset S_n$ is described in the following theorem, proved in Section~\ref{s:Gnorm1}.

\begin{theorem}\label{t:Gnorm1} Let $n\in\IN$ and $G\subset H$ be two subgroups of the symmetric group $S_n$. If for a ballean $X$ the ballean $[X]^n_G$ is normal, then the ballean $[X]^n_H$ is normal, too.
\end{theorem}

In Section~\ref{s:Gnorm2} we shall prove the following necessary conditions of the normality of $G$-symmetric powers.

\begin{theorem}\label{t:Gnorm2} Let $n\ge 2$, $G\subset S_n$ be a subgroup, and $X$ be a ballean. If the ballean $[X]^n_G$ is normal \textup{(}and $n\ge 3$\textup{)}, then the ballean $X$ has bounded growth \textup{(}and the bornology $\mathcal B_X$ of $X$ has a linearly ordered base\textup{)}.
\end{theorem}

The above results imply the following characterization of the normality of various functorial constructions over balleans.
 
\begin{theorem}\label{t:final} For a $\cof$-regular ballean $X$ the following conditions are equivalent:
\begin{enumerate}
\item for every $n\ge 1$ and every subgroup $G\subset S_n$ the ballean $[X]^n_G$ is normal;
\item for some $n\ge 3$ and some subgroup $G\subset S_n$ the ballean $[X]^n_G$  is normal;
\item for every $n\ge 1$ the ballean $X^n$  is normal;
\item for some $n\ge 2$ the ballean $X^n$  is normal;
\item the bornology $\mathcal B_X$ of $X$ has a linearly ordered base; 
\end{enumerate}
If the ballean $X$ is not ultranormal, then the conditions \textup{(1)--(5)} are equivalent to:
\begin{enumerate}
\item[(6)] for some $n\ge 2$ and some subgroup $G\subset S_n$ the ballean $[X]^n_G$ is normal;
\item[(7)] the symmetric square $[X]^2$ of $X$ is normal;
\item[(8)] the hypersymmetric square $[X]^{\le 2}$ of $X$ is normal;
\item[(9)] for some $n\ge 2$ the ballean $[X]^{\le n}$ is normal;
\item[(10)] for every $n\ge 1$ the ballean $[X]^{\le n}$ is normal;
\item[(11)] the hyperballean $[X]^{\mathcal B}$ of $X$ is normal.
\end{enumerate}
\end{theorem}

\begin{proof} The implications $(2)\Leftarrow (1)\Ra(3)\Ra(4)$ are trivial and $(2)\Ra(5)\Leftarrow(4)$ follow from Theorems~\ref{t:Gnorm2} and \ref{t:bornoreg}.  The implications $(5)\Ra(1,11)$ follows from the $\cof$-regularity of $X$ and  Propositions~\ref{p:lin-hyp}, \ref{p:Gnorm}. The implications $(1)\Ra(6,7)$ are trivial. The implications $(11)\Ra(10)\Ra(9)\Ra(8)$ trivially follow from the inclusions $[X]^2=[X]^{\le 2}\subset [X]^{\le n}\subset [X]^{\mathcal B}$ and the preservation of the normality by taking subballeans.
The condition $(6)$ implies $(2)\vee(4)\vee (7)$.

If the ballean $X$ is not ultranormal, then $(8)\Ra(5)$ by Theorem~\ref{t:hyper2}(2).
\end{proof}

Finally we present a simple example of an ultranormal ballean $X$ for which the conditions (1)--(11) of Theorem~\ref{t:final} do not hold.

\begin{example}\label{e:final} Let $X$ be an infinite set and $S_X$ be the group of permutations of $X$. Endow $X$ with the ball structure $\E_X$ consisting of the entourages $$E_F:=\big\{(x,y)\in X\times X:y\in \{x\}\cup\{f(x)\}_{f\in F}\big\}$$where $F$ runs over finite subsets of the symmetric group $S_X$. The ballean $(X,\E_X)$ is ultranormal but is not $\cof$-regular and fails to have bounded growth. Consequently, for every $n\ge 2$ and every subgroup $G\subset S_X$ the balleans $X^n$, $[X]^n_G$ and $[X]^{\le n}$ are not normal.
\end{example}

\begin{proof} The definition of the ball structure $\E_X$ implies that the bornology $\mathcal B_X$ of $X$ consists of finite subsets of $X$. Consequently, $\add(\mathcal E_X)=\add(\mathcal B_X)=\w\le\cof(\mathcal B_X)=|X|$. On the other hand, a simple diagonal argument shows that $\cof(\E_X)>|X|$ and the ballean $(X,\E_X)$ fails to have bounded growth. Since $\cof(\E_X)>\cof(\mathcal B_X)$, the ballean $(X,\E_X)$ is not $\cof$-regular. By Theorems~\ref{t:Gnorm2} and \ref{t:hyper2}(1), for every $n\ge 2$ and every subgroup $G\subset S_X$ the balleans $X^n$, $[X]^n_G$ and $[X]^{\le n}$ are not normal.
\end{proof}

The ballean constructed in Example~\ref{e:final} is a partial case of finitary balleans on $G$-spaces, which are studied in Section~\ref{s:Gspace}. In its turn, finitary balleans are partial cases of balleans generated by group ideals. An ideal $\I$ of subsets of a group $G$ is called a {\em group ideal} if for any sets $A,B\in\I$ the set $AB^{-1}$ belongs to $\I$. Each group ideal $\I$ on a group $G$ induces a bornoregular ball structure $\vec\I$ on $G$, consisting of the entourages $E_I:=\{(x,y)\in G\times G:y\in\{x\}\cup Ix\}$ where $I\in\I$. It is easy to see that the bornology of the ballean $(G,\vec \I)$ coincides with the ideal $\I$.

 For a group $G$ and an infinite cardinal $\kappa\le|G|$ by $[G]^{<\kappa}$ we denote the group ideal consisting of subsets of cardinality $<\kappa$ in $G$.

For Abelian groups and the cardinal $\kappa=\w$ the following theorem was proved by Protasov in \cite{Prot}.

\begin{theorem}\label{t:group} For any group $G$ and any infinite cardinal $\kappa<|G|$ the ballean $(G,\E_{[G]^{<\kappa}})$ is not normal.
\end{theorem} 

This theorem will be proved in Section~\ref{s:group}. It implies the following characterization, which will be proved in Section~\ref{s:group2}.

\begin{theorem}\label{t:group2} For a group $G$ and an infinite cardinal $\kappa\le|G|$ the following conditions are equivalent:
\begin{enumerate}
\item the ball structure $\E_{[G]^{<\kappa}}$ has a linearly ordered base;
\item the bornology $[G]^{<\kappa}$ of the ballean $(G,\E_{[G]^{<\kappa}})$ has a linearly ordered base;
\item the ballean $(G,\E_{[G]^{<\kappa}})$ has bounded growth;
\item $|G|=\kappa$ and $\kappa$ is a regular cardinal.
\end{enumerate}
If the cardinal $\kappa$ is regular or the group $G$ is solvable, then the conditions \textup{(1)--(4)} are equivalent to the condition
\begin{enumerate}
\item[(5)] the  ballean $(G,\E_{[G]^{<\kappa}})$ is normal.
\end{enumerate}
\end{theorem}

 In Section~\ref{s:op} we pose some open problems related to the normality of balleans.





\section{Preservation of normality by maps between balleans}\label{s:normpres}

In this section we prove some results on preservation of the normality by maps.

We say that a function $f:X\to Y$ between balleans $(X,\E_X)$ and $(Y,\E_Y)$ is
\begin{itemize}
\item {\em macro-uniform} if for any $E_X\in\E_X$ there exists $E_Y\in\E_Y$ such that $f(E_X[x])\subset E_Y[f(x)]$ for any $x\in X$;
\item {\em open} if for every $E_Y\in\E_X$ there exists $E_X\in\E_X$ such that $f(E_X[x])\supset E_Y[f(x)])$ for any $x\in X$;
\item {\em closed} if for any asymptotically disjoint sets $A,B\subset X$ with $A=f^{-1}(f(A))$ the sets $f(A)$ and $f(B)$ are asymptotically disjoint in $Y$; 
\item an {\em asymorphism} if $f$ is bijective and the functions $f$ and $f^{-1}$ are macro-uniform;
\item an {\em asymorphic embedding} if $f$ is an asymorphism of $X$ onto the subballean $f(X)$ of $(Y,\E_Y)$;
\item {\em bornologous} if for any bounded set $B\subset X$ the image $f(B)$ is bounded in $Y$;
\item {\em proper} if for any bounded set $B\subset Y$ its preimage $f^{-1}(B)$ is bounded in $X$;
\item {\em perfect} if $f$ is macro-uniform, closed and proper;
\item an {\em asymptotic immersion} if $f$ is a proper macro-uniform map such that for any asymptotically disjoint sets $A,B$ in $(X,\E_X)$ their images $f(A)$ and $f(B)$ are asymptotically disjoint in $(Y,\E_Y)$.
\end{itemize}

It is clear that each asymorphic embedding is an asymptotic immersion and each asymorphic immersion is a perfect map. An example of an asymptotic immersion which is not an asymptotic embedding will be presented in Example~\ref{ex:last}.

\begin{proposition}\label{p:emb} A ballean $X$ is normal if it admits an asymptotic immersion to a normal ballean $Y$.
\end{proposition}

\begin{proof} Given two asymptotically disjoint sets $A,B\subset X$, consider their images $f(A)$ and $f(B)$ in $Y$. Since $f$ is a asymptotic immersion, the sets $f(A)$, $f(B)$ are asymptotically disjoint in $Y$ and by the normality of $Y$ have disjoint asymptotic neighborhoods $O_{f(A)}$ and $O_{f(B)}$. Then the sets $O_A:=f^{-1}(O_{f(A)})$ and $O_B:=f^{-1}(O_{f(B)})$ are disjoint. It remains to prove that $O_A$ and $O_B$ are asymptotic neighborhoods of the sets $A,B$, respectively.

Given any $E_X\in\E_X$, find $E_Y\in\E_Y$ such that $f(E_X[x])\subset E_Y[f(x)]$ for any $x\in X$. Since $O_{f(A)}$ is an asymptotic neighborhood of $f(A)$ in the ballean $Y$, there exists a bounded set $B_Y\subset Y$ such that $E_Y[f(A)\setminus B_Y]\subset O_{f(A)}$. Since the function $f$ is proper, the preimage $B_X:=f^{-1}(B_Y)$ is a bounded set in $X$. Then for every $x\in X\setminus B_X$ we get $f(x)\in Y\setminus B_Y$ and hence $f(E_X[x])\subset E_Y[f(x)]\subset E_Y[Y\setminus B_Y]\subset O_{f(A)}$, which implies that $E_X[X\setminus B_X]\subset f^{-1}(O_{f(A)})=O_A$. This means that $O_A$ is an asymptotic neighborhood of $A$. By analogy we can prove that $O_B$ is an asymptotic neighborhood of the set $B$ in $X$.
\end{proof}

Proposition~\ref{p:emb} implies the following corollary first proved in \cite[Proposition 1.2]{Prot}.

\begin{corollary}\label{c:emb} Any subballean of a normal ballean is normal.
\end{corollary}

It is known \cite[3.7.20]{Eng} that perfect images of normal topological spaces are normal. A similar result exists also in Asymptology.

\begin{proposition}\label{p:image} A ballean $Y$ is normal if $Y$ is the image of a normal ballean $X$ under a surjective perfect map $f:X\to Y$.
\end{proposition}

\begin{proof} Assume that $f:X\to Y$ is a surjective perfect map defined on a normal ballean $X$. To show that the ballean $Y$ is normal, fix any two asymptotically disjoint sets $A,B\subset X$. Taking into account that $f$ is proper and macro-uniform, we can show that the preimages $f^{-1}(A)$ and $f^{-1}(B)$ are asymptotically disjoint sets in $X$. By the normality of $X$ the sets $f^{-1}(A)$ and $f^{-1}(B)$ has disjoint asymptotic neighborhoods $O_{f^{-1}(A)}$ and $O_{f^{-1}(B)}$, respectively.

It follows that the sets $f^{-1}(A)$ and $X\setminus O_{f^{-1}(A)}$ are asymptotically disjoint in $X$. Since the map $f$ is closed, the sets $f(f^{-1}(A))=A$ and $f(X\setminus O_{f^{-1}(A)})$ are asymptotically disjoint in $Y$ and hence $O_A=Y\setminus f(X\setminus O_{f^{-1}(A)})$ is an asymptotic neighborhood of $A$ in $Y$. It is clear that $f^{-1}(O_A)\subset O_{f^{-1}(A)}$. 

By the same reason, the set $O_B:=U\setminus f(X\setminus O_{f^{-1}(B)})$ is an asymptotic neighborhood of $B$ with $f^{-1}(O_B)\subset O_{f^{-1}(B)}$. Since the sets $O_{f^{-1}(A)}$ and $O_{f^{-1}(B)}$ are disjoint, so are the sets $O_A$, $O_B$. Therefore, the asymptotically disjoint sets $A,B$ have disjoint asymptotic neighborhoods $O_A$ and $O_B$, which means that the ballean $Y$ in normal.
\end{proof}

\begin{lemma}\label{l:open0} For a map $f:X\to Y$ and two entourages $E_X\in\E_X$ and $E_Y\in\E_Y$ the following conditions are equivalent:
\begin{enumerate}
\item $f(E^{-1}_X[x])\supset E^{-1}_Y[f(x)]$ for all $x\in X$;
\item $f^{-1}(E_Y[y])\subset E_X[f^{-1}(y)]$ for all $y\in Y$.
\end{enumerate}
\end{lemma}

\begin{proof} $(1)\Ra(2)$ Assume that $f(E_X^{-1}[x])\supset E_Y^{-1}[f(x)]$ for all $x\in X$.  Given any $y\in Y$ and  $x\in f^{-1}(E_Y[y])$, we conclude that $f(x)\in E_Y[y]$ and hence $y\in E^{-1}_Y[f(x)]\subset f(E_X^{-1}[x])$. Then $y=f(x')$ for some $x'\in E_X^{-1}[x]$ and hence $x\in E_X[x']\subset E_X[f^{-1}(y)]$. 
\smallskip

$(2)\Ra(1)$ Assume that $f^{-1}(E_Y[y])\subset E_X[f^{-1}(y)]$ for all $y\in Y$. Given any $x\in X$ and $y\in E_Y^{-1}[f(x)]$, we conclude that $f(x)\in E_Y[y]$ and hence $x\in f^{-1}(E_Y[y])\subset E_X[f^{-1}(y)]$. Then $f^{-1}(y)\cap E_X^{-1}[x]\ne\emptyset$ and $y\in f(E_X^{-1}[x])$.
\end{proof} 
 
Let $X,Y$ be sets. We say that a map $s:Y\to X$ is a {\em section} of a map $f:X\to Y$ if $f\circ s(y)=y$ for every $y\in Y$. In this case the map $f$ is surjective.

\begin{proposition}\label{p:open} A ballean $Y$ is normal if $Y$ is the image of a normal ballean $X$ under an open macro-uniform map $f:X\to Y$ that admits a bornologous section $s:Y\to X$.
\end{proposition}

\begin{proof} Assume that $f:X\to Y$ is a surjective open macro-uniform map defined on a normal ballean $X$, and let $s:Y\to X$ be a bornologous section of $f$. To show that the ballean $Y$ is normal, fix any asymptotically disjoint sets $A,B\subset X$. Put $A':=s(A)$ and $B':=f^{-1}(B)$. We claim that the sets $A'$, $B'$ are asymptotically disjoint in $X$. It suffices to prove that for any entourage $E_X\in\E_X$ the intersection $A'\cap E_X[B']$ is bounded in $X$. Since the map $f$ is macro-uniform, there exists an entourage $E_Y\in\E_Y$ such that $f(E_X[x])\subset E_Y[f(x)]$ for any $x\in X$.
Since the sets $A,B$ are asymptotically disjoint, the intersection $A\cap E_Y[B]$ is bounded and so is the set $s(A\cap E_Y[B])$.
We claim that $A'\cap E_X[B']\subset s(A\cap E_Y[B])$. Given any point $a\in A'\cap E_X[B']$, find a point $b\in B'$ with $a\in  E_X[b]$ and conclude that $$f(a)\in f(A')\cap f(E_X[b])\subset A\cap E_Y[f(b)]\subset A\cap E_Y[B]$$and hence $a\in s(A\cap E_Y[B])$. 

By the normality of $X$, the asymptotically disjoint sets $A'$ and $B'$ have disjoint asymptotic neighborhoods $O_{A'}$ and $O_{B'}$, respectively. It is easy to see that the sets  $O_A:=f(O_{A'})$ and $O_B:=Y\setminus f(X\setminus O_{A'})$ are disjoint. It remains to show that $O_A$ and $O_B$ are asymptotic neighborhoods of the sets $A$ and $B$, respectively. 

Given any entourage $E_Y\in\E_Y$ use the openness of the map $f$ to find an entourage $E_X\in\E_X$ such that $f(E_X[x])\supset E_Y[f(x)]$ and $f(E_X^{-1}[x])\supset E_Y^{-1}[f(x)]$ for every $x\in X$. By Lemma~\ref{l:open0}, $f^{-1}(E_Y[y])\subset E_X[f^{-1}(y)]$ for all $y\in Y$.

Since $O_{A'}$ is an asymptotic neighborhood of $A'$, the set $D=\{a\in A':E_X[a]\not\subset O_{A'}\}$ is bounded in $X$. Then the set $f(D)$ is bounded in $Y$ and for every $a\in A\setminus f(D)$ we have $s(a)\in A'\setminus D$ and hence $E_X[s(a)]\subset O_{A'}$, which implies $E_Y[a]=E_Y[f(s(a))]\subset f(E_X[s(a)])\subset f(O_{A'})=O_A$.

Since $O_{B'}$ is an asymptotic neighborhood of $B'=f^{-1}(B)$, the set $D'=\{b\in B':E_X[b]\not\subset O_{B'}\}$ is bounded in $X$ and the set $f(D')$ is bounded in $Y$. We claim that $E_Y[b]\subset O_B$ for every point $b\in B\setminus f(D')$. Taking into account that $f^{-1}(b)\subset B'\setminus D'$, we conclude that $E_X[f^{-1}(b)]\subset O_{B'}$ and hence $f^{-1}(E_Y[b])\subset E_X[f^{-1}(b)]\subset O_{B'}$ which implies $E_Y[b]\subset O_B$, by the definition of the set $O_B$. This means that $O_B$ is an asymptotic neighborhood of $B$, and the ballean $Y$ is normal.
\end{proof}

\section{Balleans of bounded growth}\label{s:growth}

In this section we study balleans of bounded growth and prove Theorem~\ref{t:main2}. More precisely, this theorem follows from Corollary~\ref{c:g} and Lemmas~\ref{l:g1}, \ref{l:g2}.
  
We recall that a ballean $X$ has {\em bounded growth} if there exists a set $E\subset X\times X$ containing $\Delta_X$ (and called a {\em growth entourage} of $X$) such that for any bounded set $B\subset X$ the set $G[B]$ is bounded and for any entourage $E\in\E_X$ there exists a bounded subset $B\subset X$ such that $E[x]\subset G[x]$ for all $x\in X\setminus B$.

For a ballean $(X,\E_X)$ denote by ${\Uparrow}{\mathcal B_X}$ the family of all subsets $G\subset X\times X$ such that $\Delta_X\subset G$ and for any bounded set $B\subset X$ the sets $G[B]$ and $G^{-1}[B]$ are bounded in $(X,\E_X)$. Observe that ${\Uparrow}{\mathcal B_X}$ is a coarse structure of $X$, containing the ball structure $\E_X$. The coarse structure ${\Uparrow}\mathcal B_X$ is called the {\em universal coarse structure} of the bornology $\mathcal B_X$.

Let $\cof_*(\E_X,{\Uparrow}{\mathcal B_X})$ be the smallest cardinality $|\U|$ of a subfamily $\U\subset {\Uparrow}{\mathcal B_X}$ such that for every $E\subset \E_X$ there exists $U\in\U$ such that $E\setminus U$ is bounded in $X\times X$.

\begin{proposition} A ballean $X$ has bounded growth if and only if $\cof_*(\E_X,{\Uparrow}\mathcal B_X)=1$.
\end{proposition} 
  
\begin{proof}The ``if'' part follows from the definitions. To prove the ``only if'' part, assume that $G$ is growth entourage for the ballean $(X,\E_X)$. Since $G[B]\in\mathcal B_X$ for all $B\in\mathcal B_X$, the entourage $S=G\cap G^{-1}$ belongs to the universal coarse structure ${\Uparrow}\mathcal B_X$ of the bornology $\mathcal B_X$. By the choice of $G$, for any entourage $E\in\E_X$ there exists a bounded set $B\subset X$ such that $E[x]\cup E^{-1}[x]\subset G[x]$ for all $x\in X\setminus B$. Then the complement $E\setminus S$ is contained in the bounded set $E^{-1}[B]\times E{\circ} E^{-1}[B]$ and hence $\cof_*(\E_X,{\Uparrow}\mathcal B_X)\le|\{S\}|=1$.
\end{proof}

Let us observe that the bounded growth of balleans nicely interacts with proper macro-uniform maps.

\begin{proposition}\label{p:map} A ballean $X$ has bounded growth if $X$ admits a proper macro-uniform map $f:X\to Y$ to a ballean $Y$ of bounded growth.
\end{proposition}

\begin{proof} Assume that $Y$ has bounded growth and let $G_Y$ be a growth entourage for $Y$. It is easy to check that the entourage $G_X:=\{(x,y)\in X\times X:(f(x),f(y))\in G_Y\}$ is a growth entourage for the ballean $(X,\E_X)$.
\end{proof}

Proposition~\ref{p:map} implies that the bounded growth is preserved by taking subballeans.

\begin{corollary}\label{c:g} If a ballean $X$ has bounded growth, then each subballean of $X$  has bounded growth, too.
\end{corollary}

\begin{proposition}\label{p:lg} A  ballean $X$ has bounded growth if its coarse structure has a linearly ordered base.
\end{proposition}

\begin{proof} Assume that the coarse structure ${\downarrow}\E_X$ of a ballean $X$ has a linearly ordered base. Then 
$\add({\downarrow}\E_X)=\cof({\downarrow}\E_Y)=\kappa$ is a regular cardinal and we can choose a well-ordered base $\{E_\alpha\}_{\alpha\in\kappa}\subset {\downarrow}\E_X$ of the coarse structure ${\downarrow}\E_X$ such that
\begin{itemize}
\item $E_0=\emptyset$;
\item $E_\alpha=\bigcup_{\beta<\alpha}E_\beta$ for any limit ordinal $\alpha\in\kappa$, and
\item $E_\alpha\subset E_{\alpha+1}$ for any ordinal $\alpha\in\kappa$.
\end{itemize} 

Fix any point $x_0\in X$ and consider the entourage $$G:=\bigcup_{\alpha\in\kappa}(E_{\alpha+1}[x_0]\setminus E_\alpha[x_0])\times E_\alpha[x_0].$$
It is easy to check that $G$ is a growth entourage for the ballean $X$.
\end{proof}   
 
\begin{lemma}\label{l:g1} The product $X\times Y$ of two balleans has bounded growth if the balleans $X,Y$ have bounded growth and the bornology $\mathcal B_{X\times Y}$ has a linearly ordered base.
\end{lemma}

\begin{proof} Assume that the balleans $X,Y$ have bounded growth and the bornology $\mathcal B_{X\times Y}$ of $X\times Y$ has a linearly ordered base.
Then for the regular cardinal $\kappa=\add(\mathcal B_{X\times Y})=\cof(\mathcal B_{X\times Y})$ we can choose a well-ordered base $\{B_\alpha\}_{\alpha\in\kappa}$ of the bornology $\mathcal B_{X\times Y}$.

For every $\alpha\in\kappa$ denote by $C_\alpha$ and $D_\alpha$ the projection of the bounded set $B_\alpha$ on $X$ and $Y$, respectively. Then $(C_\alpha\times D_\alpha)_{\alpha\in\kappa}$ is a well-ordered base of the bornology $\mathcal B_{X\times Y}$ and we can assume that $B_\alpha=C_\alpha\times D_\alpha$.
Also we can assume that $B_0=\emptyset$ and $B_\alpha=\bigcup_{\beta<\alpha}B_\beta$ for all limit ordinals $\alpha<\kappa$.

Let $G_X$ and $G_Y$ be growth entourages of the ballean $X$ and $Y$, respectively. 

We claim that the entourage 
$$G:=\bigcup_{\alpha<\kappa}(B_{\alpha+1}\setminus B_\alpha)\times (G_X[C_{\alpha+1}]\times G_Y[D_{\alpha+1}])$$ witnesses that the ballean $X\times Y$ has bounded growth. Given a bounded subset $B\subset X\times Y$ we can find an ordinal $\alpha\in\kappa$ with $B\subset B_{\alpha+1}$ and conclude that the set $G[B]\subset G[B_{\alpha+1}]=G_X[C_{\alpha+1}]\times G_Y[D_{\alpha+1}]$ is bounded in $X\times Y$.

Next, given any entourage $E\in\E_{X\times Y}$, find entourages $E_X\in\E_X$ and $E_Y\in\E_Y$ such that $E[(x,y)]=E_X[x]\times E_Y[y]$ for any $(x,y)\in X\times Y$.
For the entourages $E_X$ and $E_Y$ find bounded sets $B_X\in\mathcal B_X$ and $B_Y\in\mathcal B_Y$ such that $E_X[x]\subset G_X[x]$ and $E_Y[y]\subset G_Y[y]$ for any points $x\in X\setminus B_X$ and $y\in Y\setminus B_Y$. Finally choose an ordinal $\alpha\in\kappa$ such that $E_X[B_X]\times E_Y[B_Y]\subset B_\alpha=C_\alpha\times D_\alpha$. Then $E_X[x]\subset G_X[x]\cup C_\alpha$ and $E_Y[y]\subset G_Y[y]\cup D_\alpha$ for all $x\in X$ and $y\in Y$.

We claim that $E[(x,y)]\subset G[(x,y)]$ for any $(x,y)\in (X\times Y)\setminus B_\alpha$. Given any such pair $(x,y)$, choose a unique ordinal $\beta\in\kappa$ such that $(x,y)\in B_{\beta+1}\setminus B_\beta$ and observe that $\beta\ge \alpha$.
Then 
$$
\begin{aligned}
E[(x,y)]&=E_X[x]\times E_Y[y]\subset (G_X[x]\cup C_\alpha)\times (G_Y[y]\cup D_\alpha)\subset\\
&\subset (G_X[C_{\beta+1}]\cup C_\alpha)\times (G_Y[D_{\beta+1}]\cup D_\alpha)=G_X[C_{\beta+1}]\times G_Y[D_{\beta+1}]=G[(x,y)].
\end{aligned}
$$
\end{proof}

\begin{lemma}\label{l:g2} If the product $X\times Y$ of two unbounded balleans has bounded growth, then its bornology $\mathcal B_{X\times Y}$ has a linearly ordered base.
\end{lemma}

\begin{proof} Let $G$ be a growth entourage for the ballean $X\times Y$. We lose no generality assuming that $\add(\mathcal B_X)\le \add(\mathcal B_Y)$.  In this case we shall show that $\cof(\mathcal B_Y)\le\add(\mathcal B_X)$. 

By the definition of the cardinal $\kappa:=\add(\mathcal B_X)$, there exists a transfinite sequence $(B_\alpha)_{\alpha\in\kappa}$ of bounded sets in $X$ whose union $\bigcup_{\alpha\in\kappa}B_\alpha$ is not bounded in $X$.  Fix any point $y_0\in Y$ and for every $\alpha\in\kappa$ consider the projection $D_\alpha$ of the bounded set $G[B_\alpha\times\{y_0\}]$ onto $Y$.
We claim that the family $\{D_\alpha\}_{\alpha\in\kappa}$ is cofinal in $\mathcal B_Y$. Indeed, given any bounded set $B\subset Y$, find an entourage $E_Y\in\E_Y$ such that $B\subset E_Y[y_0]$. Take any entourage $E_X\in\E_X$ and consider the entourage $E\in\E_{X\times Y}$ such that $E[(x,y)]=E_X[x]\times E_Y[y]$ for all $(x,y)\in X\times Y$. Since $G$ is a growth entourage for $X\times Y$, for the entourage $E$ there exists a bounded set $D\subset X\times Y$ such that $E[(x,y)]\subset G[(x,y)]$ for all $(x,y)\in (X\times Y)\setminus D$. Since the union $\bigcup_{\alpha\in\kappa}B_\alpha$ is not bounded in $X$, there exists $\alpha\in\kappa$ such that  $B_\alpha\times\{y_0\}\not\subset D$. Choose any $x\in B_\alpha$ with $(x,y_0)\notin D$ and observe that 
$\{x\}\times B\subset \{x\}\times E_Y[y_0]\subset E[(x,y_0)]\subset G[(x,y_0)]\subset G[B_\alpha\times\{y_0\}]\subset X\times D_\alpha$ and hence $B\subset D_\alpha$. This completes the proof of the inequality $\cof(\mathcal B_Y)\le\add(\mathcal B_X)$.

Then $\add(\mathcal B_Y)\le  \cof(\mathcal B_Y)\le\add(\mathcal B_X)\le\add(\mathcal B_Y)$ and hence $\add(\mathcal B_Y)=\cof(\mathcal B_Y)=\add(\mathcal B_X)$. Since $\add(\mathcal B_X)\le\add(\mathcal B_Y)$, we can repeat the above argument and prove that $\cof(\mathcal B_X)\le\add(\mathcal B_Y)$, which implies that  $\add(\mathcal B_Y)=\cof(\mathcal B_Y)=\add(\mathcal B_X)=\cof(\mathcal B_X)=\kappa$ for some regular cardinal $\kappa$.

It follows that the bornologies $\mathcal B_X$ and $\mathcal B_Y$ have well-ordered bases $\{C_\alpha\}_{\alpha\in\kappa}\subset\mathcal B_X$ and $\{D_\alpha\}_{\alpha\in\kappa}\subset\mathcal B_Y$. Then $\{C_\alpha\times D_\alpha\}_{\alpha\in\kappa}$ is a well-ordered base for the bornology $\mathcal B_{X\times Y}$ of the ballean $X\times Y$.
\end{proof}

\begin{proposition}\label{p:discrete} If an unbounded ballean $(X,\E_X)$ has bounded growth, then it contains a discrete subballean.
\end{proposition}

\begin{proof} Let $G$ be a growth entourage for the ballean $(X,\mathcal B_X)$.
Let $\mathcal S$ be the family of subsets $S\subset X$ endowed with a well-order $<_S$ such that $x\notin \bigcup_{y<_Sx}G[y]$ for any $x\in S$. The family $\mathcal S$ is endowed with the partial order $\le$ defined by $(S,<_S)\le (W,<_W)$ iff $(S,<_S)$ is an initial interval of $(W,<_W)$. It is easy to see that each chain in $\mathcal S$ is upper bounded. So, by the Kuratowski--Zorn Lemma, the poset $\mathcal S$ has a maximal element $(M,<_M)$. 

We claim that $M$ is a discrete subballean of $X$. First we show that $M$ is unbounded. Assuming that $M$ is bounded, we can choose a point $y\in X\setminus \bigcup_{x\in M}G[x]$ and consider the set $S:=M\cup\{y\}$ endowed with the well-order $<_S$ such that $<_M\subset <_S$ and $x<_S y$ for all $x\in M$. It follows that the well-ordered set $(S,<_S)$ belongs to $\mathcal S$ and is strictly larger than $(M,<_M)$, which contradicts the maximality of $M$. This contradiction shows that the set $M$ is unbounded. Next, we show that $M$ is discrete. By the definition of growth entourage  $G$, for any entourage $E\in\E_X$ there exists a bounded set $B\subset X$ such that $E^{-1}{\circ}E[x]\subset G[x]$ for any $x\in X\setminus B$. We claim that $E[x]\cap E[y]=\emptyset$ for any distinct points $x,y\in M\setminus B$. Without loss of generality, $x<_M y$ and hence $y\notin G[x]$. Since $E^{-1}{\circ}E[x]\subset G[x]$, we conclude that $y\notin E^{-1}{\circ}E[x]$ and hence $E[y]\cap E[x]=\emptyset$.
\end{proof}

\section{A characterization of normality and proof of Theorem~\ref{t:prot2}}\label{s:prot2}

In this section we prove Theorem~\ref{t:prot2}, but start with the following characterization of normality.

\begin{proposition}\label{p:normchar} A ballean $(X,\E_X)$ is normal if and only if any asymptotically disjoint sets $A,B\subset X$ have asymptotically disjoint asymptotic neighborhoods $O_A,O_B$. 
\end{proposition} 

\begin{proof} The ``if'' part trivially follows from the observation that for any asymptotically disjoint asymptotic neighborhoods $O_A$, $O_B$ of sets $A,B$ in $X$ the sets $U_A:=O_A$ and $U_B:=O_B\setminus O_A$ are disjoint asymptotic neighborhoods of $A$ and $B$, respectively.
\smallskip

 To prove the ``only if'' part, assume that the ballean $X$ is normal. Fix any asymptotically disjoint sets $A,B\subset X$ and find disjoint asymptotic neighborhoods $O_A$ and $O_B$ of $A$ and $B$, respectively. Since the sets $A$ and $X\setminus O_A$ are asymptotically disjoint, we can use the normality of $X$ once more and find two disjoint sets $U_A,U_B$ such that $U_A$ is an asymptotic neighborhood of $A$ and $U_B$ is an asymptotic neighborhood of $X\setminus O_A$. We claim that the sets $U_A$ and $O_B$ are asymptotically disjoint. Since $U_B$ is an asymptotic neighborhood of $X\setminus O_A$, for any entourage $E\in\E_X$ there exists a bounded set $D\subset X$ such that $E^{-1}\circ E[x]\subset U_B$ for any $x\in (X\setminus O_A)\setminus D$. We claim that $E[U_A]\cap E[O_B]\subset E[D]$. To derive a contradiction, assume that some point $x\in E[U_A]\cap E[O_B]$ does not belong to $E[D]$. Choose points  $a\in U_A$ and $b\in O_B$ with $x\in E[a]\cap E[b]$. It follows from $x\notin E[D]$ that $b\notin D$. Taking into account that $b\in O_B\subset X\setminus O_A$, we conlcude that that $E^{-1}\circ E[b]\subset U_B$. Then $a\in E^{-1}[x]\subset E^{-1}\circ E[b]\subset U_B$, which is not possible as $a\in U_A\subset X\setminus U_B$.
 \end{proof}

Now we present a {\em proof of Theorem~\ref{t:prot2}}. Assume that a ballean $X$ has $\cof_\star(\E_X)\le \add(\mathcal B_X)$.

To prove that the ballean $(X,\E_X)$ is normal, fix any two asymptotically disjoint sets $A,B\subset X$. By the definition of the cardinal $\kappa=\cof_\star(\E_X)$, for the sets $A,B$ there exists a subfamily $\{E_\alpha\}_{\alpha\in\kappa}\subset\E_X$ such that for any $E\in\E_X$ there are ordinals $\alpha,\beta\in\kappa$ such that the sets $E[A]\setminus E_\alpha[A]$ and $E[B]\setminus E_\beta[B]$ are bounded in $(X,\E_X)$.

 It is clear that the sets 
$$O_A=\bigcup_{\alpha\in\kappa}\big(E_\alpha[A]\setminus\bigcup_{\beta\le\alpha}E_\beta[B]\big)\mbox{ and }O_B=\bigcup_{\beta\in\kappa}\big(E_\beta[B]\setminus\bigcup_{\alpha\le\beta}E_\alpha[A]\big)$$are disjoint.
We claim that $O_A$ and $O_B$ are asymptotic neighborhoods of the sets $A$ and $B$, repectively. Given any entourage $E\in\E_X$, find $\alpha\in \kappa$ such that the set $E[A]\setminus E_\alpha[A]$ is bounded. Since $\alpha<\kappa\le\add(\mathcal B_X)$ the union $U=\bigcup_{\beta\le\alpha}(E_\alpha[A]\cap E_\beta[B])$ belongs to the bornology $\mathcal B_X$. It is easy to check that $E[A]\setminus O_A\subset (E[A]\setminus E_\alpha[A])\cup U$. So $E[A]\setminus O_A$ is a bounded set and $O_A$ is an asymptotic neighborhood of $A$ in the ballean $X$.

By analogy we can prove that the set $O_B$ is an asymptotic neighborhood of $B$. Therefore, the asymptotically disjoint sets $A,B$ have disjoint asymptotic neighborhoods and the ballean $X$ is normal.

\section{Proof of Theorem~\ref{t:main}}\label{s:proofs}

Theorem~\ref{t:main} follows from Lemmas~\ref{l1}--\ref{l3}, proved in this section.

\begin{lemma}\label{l1} If the product $X\times Y$ of two balleans is normal, then there exist functions $\varphi:\mathcal B_X\to\mathcal B_Y$ and $\psi:\mathcal B_Y\to\mathcal B_X$ such that for any sets $A\in\mathcal B_X$, $B\in\mathcal B_Y$ either $A\subset \psi(B)$ or $B\subset\varphi(A)$.
\end{lemma}

\begin{proof} Fix any point $(x_0,y_0)\in X\times Y$ and consider the subsets $V:=\{x_0\}\times Y$ and $H:=X\times \{y_0\}$ of the product. By definition of the ball structure of $X\times Y$, for every entourage $E\in \E_{X\times Y}$ there exist entourages $E_1\in\E_X$ and $E_2\in\E_Y$ such that $E[(x,y)]=E_1[x]\times E_2[y]$ for any $x,y\in X$. Then $E[V]=E_1[x_0]\times Y$ and $E[H]=X\times E_2[y_0]$, which implies that the intersection $E[V]\cap E[H]=E_1[x_0]\times E_2[y_0]$ is a bounded subset of $X\times Y$.
This means that the sets $V,H$ are asymptotically disjoint. By the normality of $X\times Y$, these sets have disjoint asymptotic neighborhoods $O_V,O_H\subset X\times Y$.

 Since $O_V$ is an asymptotic neighborhood of $V$, there exists a function  $f:\E_{X}\to \mathcal B_{Y}$ assigning to each entourage $E_1\in\E_X$ a bounded set $f(E_1)\subset Y$ such that $E_1[x_0]\times \{y\}\subset O_V$ for all $y\notin f(E_1)$. By analogy, there exists a function $g:\E_{Y}\to \mathcal B_{X}$ assigning to each entourage $E_2\in\E_Y$ a bounded set $g(E_2)\subset X$ such that $\{x\}\times E_2[y]\subset O_H$ for every $x\notin g(E_2)$. 

For every bounded sets $A\in\mathcal B_X$ and $B\in\mathcal B_Y$ choose entourages $E_A\in\mathcal E_X$ and $E_B\in\mathcal E_Y$ such that $A\subset E_A[x_0]$ and $B\subset E_B[y_0]$.

Finally, consider the functions $$\varphi:\mathcal B_X\to\mathcal B_Y,\;\;\varphi:A\mapsto f(E_A),\mbox{ \ and \ }\psi:\mathcal B_Y\to\mathcal B_X,\;\;\psi:B\mapsto g(E_B).$$ We claim that these functions have the required property. Indeed, for any $A\in\mathcal B_X$ and $B\in\mathcal B_Y$ we have $$
\begin{aligned}
&\big(A\times (Y\setminus \varphi(A))\cap ((X\setminus \psi(B))\times B\big)\subset\\
&\subset \big(E_A[x_0]\times (Y\setminus f(E_A)\big)\cap\big((X\setminus g(E_B))\times E_B[y_0]\big)\subset O_V\cap O_H=\emptyset
\end{aligned}
$$ and hence $A\cap(X\setminus \psi(B))=\emptyset$ or $(Y\setminus \varphi(A))\cap B=\emptyset$, which is equivalent to $A\subset \psi(B)$ or $B\subset \varphi(A)$.
 \end{proof}
 
 \begin{lemma}\label{l2} Assume that for two unbounded partially ordered sets $P,Q$ there exist functions $\varphi:P\to Q$ and $\psi:Q\to P$ such that for any $x\in X$ and $y\in Y$ either $x\le\psi(y)$ or $y\le\varphi(x)$. Then $\add(P)=\cof(P)=\cof(Q)=\add(Q)$.
 \end{lemma}
 
 \begin{proof} Without loss of generality, $\add(P)\le\add(Q)$. By the definition of the cardinal $\add(P)$, there exists an unbounded set $\{p_\alpha\}_{\alpha\in\add(P)}$ in $P$ of cardinality $\add(P)$.
 
 We claim the set $\{\varphi(p_\alpha)\}_{\alpha\in\add(P)}$ is cofinal in $Q$. Given any $q\in Q$, consider the element $\psi(q)\in P$ and find $\alpha\in\add(P)$ such that $p_\alpha\not\le \psi(q)$. Then our assumption guarantees that $q\le\varphi(p_\alpha)$. Therefore, the set $\{\varphi(p_\alpha)\}_{\alpha\in\add(P)}$ is cofinal  in $Q$ and hence $\cof(Q)\le\add(P)$. 
 
 It follows that $\add(Q)\le\cof(Q)\le\add(P)\le\add(Q)$ and hence $\add(Q)=\add(P)=\cof(Q)\le\cof(P)$. Taking into account that $\add(Q)\le\add(P)$, we can repeat the above argument and prove that $\cof(P)\le\add(Q)$. Consequently,
 $\cof(P)=\add(P)=\add(Q)=\cof(Q)$.
 \end{proof}
  
\begin{lemma}\label{l3} If the product $X\times Y$ of two unbounded balleans $X,Y$ is normal, then the balleans $X$ and $Y$  have bounded growth.
\end{lemma}

\begin{proof} By Lemmas~\ref{l1} and \ref{l2}, $\add(\mathcal B_X)=\cof(\mathcal B_X)= \add(\mathcal B_Y)=\cof(\mathcal B_Y)=\kappa$ for some infinite regular cardinal $\kappa$. This implies that the bornologies $\mathcal B_X$ and $\mathcal B_Y$ have bases $\{C_\alpha\}_{\alpha\in\kappa}\subset \mathcal B_X$ and $\{D_\alpha\}_{\alpha\in\kappa}\subset\mathcal B_Y$ such that
\begin{itemize}
\item $C_0=\emptyset$ and $D_0=\emptyset$;
\item $C_\alpha\subset C_\beta\subset X$ and $D_\alpha\subset D_\beta\subset Y$ for any $\alpha<\beta<\kappa$;
\item $C_{\alpha+1}\ne C_\alpha$ and $D_{\alpha+1}\ne D_\alpha$ for every ordinal $\alpha\in\kappa$;
\item $C_\alpha=\bigcup_{\beta<\alpha}C_\beta$ and $D_\alpha=\bigcup_{\beta<\alpha}D_\beta$ for any limit ordinal $\alpha<\kappa$.
\end{itemize}

Fix any point $x_0\in X$ and consider the sets $A=\{x_0\}\times Y$ and 
$$B:=\bigcup_{\alpha<\kappa}(X\setminus C_\alpha)\times(D_{\alpha+1}\setminus D_\alpha)$$
in $X\times Y$. We claim that these sets are asymptotically disjoint. Given any entourage $E\in\E_{X\times Y}$, we need to check that the intersection $E[A]\cap E[B]$ is bounded. By definition of the ball structure $\E_{X\times Y}$ there exist two entourages $E_1\in\mathcal E_X$ and $E_2\in\mathcal E_Y$ such that $E[(x,y)]=E_1[x]\times E_2[y]$ for any $(x,y)\in X\times Y$. 

Observe that the set $E_1^{-1}[E_1[x_0]]\subset X$ is bounded  and hence is contained in some bounded set $C_\alpha$, $\alpha<\kappa$. We claim that $E[A]\cap E[B]\subset E[C_\alpha\times D_{\alpha}]$. Given any point $(x,y)\in E[A]\cap E[B]$, find points $a=(x_0,a_2)\in A$ and $b=(b_1,b_2)\in B$ such that $$(x,y)\in E[a]\cap E[b]=(E_1[x_0]\times E_1[a_2])\cap (E_1[b_1]\times E_2[b_2]).$$
It follows from $x\in E_1[x_0]\cap E_1[b_1]$ that $b_1\in E^{-1}[E_1[x_0]]\subset C_\alpha$. Find a unique ordinal $\beta<\alpha$ such that $b_2\in D_{\beta+1}\setminus D_\beta$.  The definition of the set $B$ guarantees that $b_1\in X\setminus C_\beta$ and hence $x\in C_\alpha\setminus C_\beta$, which implies that $\beta<\alpha$ and $(b_1,b_2)\in C_\alpha\times D_\alpha$. Finally, $x\in E[b]\subset E[C_\alpha\times D_\alpha]$ and hence the set $E[A]\cap E[B]$ is bounded.

By the normality of the ballean $X\times Y$, the asymptotically disjoint sets $A$ and $B$ have disjoint asymptotic neighborhoods $O_A$ and $O_B$. Then for every $\alpha\in\kappa$ we can find an ordinal $f(\alpha)\in [\alpha,\kappa)$ such that $C_{\alpha+1}\times (Y\setminus D_{f(\alpha)})\subset O_A$. Since the cardinal $\kappa$ is regular, we can assume that the function $f:\kappa\to\kappa$ is increasing.

We claim that the entourage
$$G_Y:=\bigcup_{\alpha\in\kappa}(D_{\alpha+1}\setminus D_\alpha)\times D_{f(\alpha)}$$ witnesses that the ballean $Y$ has bounded growth.

Given any bounded set $D\subset Y$, find $\alpha\in\kappa$ such that $D\subset D_{\alpha+1}$ and conclude that $G_Y[D]\subset G_Y[D_{\alpha+1}]\subset D_{f(\alpha)}$, which means that the set $G_Y[D]$ is bounded in $Y$.

Next, given any entourage $E_Y\in\mathcal E_Y$ we can find an ordinal $\alpha\in\kappa$ such that $\{b_1\}\times E_Y[b_2]\subset O_B$ for any point $(b_1,b_2)\in B\setminus (C_\alpha\times D_\alpha)$. We claim that for any point $y\in Y\setminus D_\alpha$ we have  $E_Y[y]\subset G_Y[y]$. Given any $y\in Y\setminus D_\alpha$, find a unique ordinal $\beta\ge \alpha$ such that $y\in D_{\alpha+1}\setminus D_\alpha$ and choose any point $x\in C_{\alpha+1}\setminus C_\alpha$. The definition of the set $B$ guarantees that the pair $(x,y)$ belongs $B$ and hence the set $\{x\}\times E_Y[y]\subset O_B$ is disjoint with the asymptotic neighborhood $O_A$ of $A$. 
Since $\{x\}\times (Y\setminus D_{f(\alpha)})\subset O_A$, the sets $E_Y[y]$ and $Y\setminus D_{f(\alpha)}$ are disjoint, which means that $E_Y[y]\subset D_{f(\alpha)}=G_Y[y]$.

Therefore, $G_Y$ is a growth entourage, witnessing that the ballean $Y$ has bounded growth. By analogy we can prove that the ballean $X$ has bounded growth.
\end{proof}

\section{Discrete and ultradiscrete balleans}

Let us recall that a ballean $X$ is {\em discrete} if for any entourage $E\in\E_X$ there exists a bounded set $B_E\subset X$ such that $E[x]=\{x\}$ for all $x\in X\setminus B_E$. 

Each discrete ballean $(X,\E_X)$ has bounded growth as $$\cof_*(\E_X,{\Uparrow}\mathcal B_X)\le\cof_*(\E_X)\le|\{\Delta_X\}|=1.$$ Since $\cof_\star(\E_X)=\cof_*(\E_X)=1\le\add(\mathcal B_X)$, discrete balleans are normal, according to Theorem~\ref{t:prot2}.

\begin{proposition}\label{p:ultra1} For a discrete ballean $X$ the following conditions are equivalent:
\begin{enumerate}
\item the ballean $X$ is ultranormal;
\item the bornology $\mathcal B_X$ of $X$ is a maximal ideal on $X$;
\item the family $\{X\setminus B:B\in\mathcal B_X\}$ is an ultrafilter on $X$.
\end{enumerate}
\end{proposition}

\begin{proof} The equivalence $(2)\Leftrightarrow(3)$ is trivial. 
\smallskip

To prove that $(1)\Ra(2)$, assume that the bornology $\mathcal B_X$ of $X$ is not a maximal ideal on $X$. Then there exists an unbounded set $A\subset X$ such that $A\cup B\ne X$ for all $B\in\mathcal B_X$. In particular, $X\setminus A\notin\mathcal B_X$, which means that $A$ and $X\setminus A$ are two disjoint unbounded sets in $X$. Since the ballean $X$ is discrete, these disjoint sets are asymptotically disjoint in $X$. Therefore the ballean $X$ is not ultranormal.
\smallskip

To prove that $(2)\Ra(1)$, assume that the ballean $X$ is not ultranormal. 
Then $X$ contains two asymptotically disjoint unbounded sets $A,A'\subset X$. Then  $$\I=\{I\subset X:\exists B\in\mathcal B_X\;\;I\subset A\cup B\}\supset\mathcal B_X$$ is an ideal of sets on $X$, witnessing that the bornology $\mathcal B_X$ is not a maximal ideal on $X$.
\end{proof}

We recall that a ballean is {\em ultradiscrete} if it is discrete and ultranormal.

\begin{proposition}\label{p:ultra} For any ultradiscrete ballean $(X,\E_X)$ we have
$\add(\mathcal B_X)<\cof(\mathcal B_X)$.
\end{proposition}

\begin{proof}  Assuming that $\cof(\mathcal B_X)=\add(\mathcal B_X)$, we conclude that the bornology $\mathcal B_X$ has a well-ordered base $(B_\alpha)_{\alpha\in\kappa}$ of cardinality $\kappa=\add(\mathcal B_X)=\cof(\mathcal B_X)$. Replacing $(B_\alpha)_{\alpha\in\kappa}$ by a cofinal subsequence, we can assume that for every $\alpha<\kappa$ the set $B_{\alpha+1}\setminus B_\alpha$ contains two distinct points $y_\alpha,z_\alpha$. Then $Y:=\{y_\alpha\}_{\alpha<\kappa}$ and $Z=\{z_\alpha\}_{\alpha<\kappa}$ are two disjoint unbounded sets in $X$. Since the ballean $X$ is discrete, the disjoint sets $Y,Z$ are asymptotically disjoint, which is not possible as $(X,\E_X)$ is ultranormal.
\end{proof}

\begin{proposition}\label{p:ultra3n} For any ultradiscrete ballean $(X,\E_X)$ its coarse structure ${\downarrow}\E_X$ coincides with the universal coarse structure ${\Uparrow}\mathcal B_X$ of its bornology $\mathcal B_X$.
\end{proposition}

\begin{proof} It is clear that ${\downarrow}\E_X\subset{\Uparrow}\mathcal B_X$. Assuming that ${\downarrow}\mathcal E_X\ne{\Uparrow}\mathcal B_X$, we can find an entourage $E=E^{-1}\in {\Uparrow}\mathcal B_X\setminus{\downarrow}\mathcal E_X$. Since $E\notin {\downarrow}\mathcal E_X$, the set $B=\{x\in X:\{x\}\ne E[x]\}$ is unbounded. Let $C\subset B$ be a maximal set such that $E[x]\cap E[y]=\emptyset$ for any distinct points $x,y\in C$. The maximality of $C$ in the unbounded set $B$ implies that the set $C$ is unbounded.

By the definition of the set $B\supset C$ there exists a function $f:C\to X$ such that $f(x)\in E[x]\setminus\{x\}$ for all $x\in C$. The choice of $C$ guarantees that the sets $C$ and $f(C)$ are disjoint. Since $\mathcal B_X$ is a maximal ideal (by Proposition~\ref{p:ultra}), $C\notin\mathcal B_X$ implies that $X\setminus C\supset f(C)$ belongs to $\mathcal B$ and hence is bounded. Then $C\subset E^{-1}[f(C)]$ also is bounded, which is a desired contradiction.
\end{proof}

\begin{example}\label{ex:ultra} There exists a discrete ballean $(X,\mathcal E_X)$ which is not ultranormal but has $\E_X={\Uparrow}\mathcal B_X$.
\end{example}

\begin{proof} Using the Kuratowski-Zorn Lemma, enlarge the ideal $[X]^{\le\w}$ of at most countable sets on the ordinal $X=\w_1$ to a maximal ideal $\I$ on $X$. Fix also any maximal ideal $\J$ on the ordinal $\w\subset \w_1$. 

Consider the subideal $\mathcal B_X=\{B\in \I:B\cap\w\in\J\}$ of the ideal $\I$ and let $\E_X$ be the discrete coarse structure on $X=\w_1$ generated by the base $\{(B\times B)\cup\Delta_X:B\in\mathcal B_X\}$. The discrete coarse space $(X,\E_X)$ is not ultradiscrete since its bornology $\mathcal B_X$ is not a maximal ideal on $X$.

We claim that ${\Uparrow}\mathcal B_X=\E_X$. In the opposite case, we can fix an entourage $E\in {\Uparrow}\mathcal B_X\setminus\E_X$. Replacing $E$ by $E\cup E^{-1}$ we can assume that $E=E^{-1}$. For this entourage the set $A=\{x\in X:\{x\}\ne E[x]\}$ does not belong to the bornology $\mathcal B_X$. Using the Kuratowski-Zorn's Lemma, choose a maximal subset $A'\subset A$ such that $E[x]\cap E[y]=\emptyset$ for any distinct points $x,y\in A'$. By the maximality, the set $A'$ is unbounded in $(X,\E_X)$. 

By the definition of the set $A\supset A'$, there exists a function $f:A'\to X$ such that $f(x)\in E[x]\setminus \{x\}$ for every $x\in A'$. The property of the set $A'$ implies that the function $f$ is injective and its image $f(A')$ is disjoint with $A'$.
Since the set $A'\subset E^{-1}[f(A')]$ is unbounded, the set $f(A')$ is unbounded too. Let $M\subset A'$ be a maximal subset such that the family $(E[f(x)])_{x\in M}$ is disjoint. The maximality of $M$ and unboundedness of $f(A')$ imply that the set $f(M)\subset E[M]$ is unbounded and so is the set $M$.

Since each of the  families $(E[x])_{x\in M}$ and $(E[f(x)])_{x\in M}$ is disjoint, there exists a countable subset $M_0\subset M$ such that $\w\cap E[x]=\emptyset=\w\cap E[f(x)]$ for all $x\in M\setminus M_0$. 
Then $M\setminus M_0$ and $f(M\setminus M_0)$ are two disjoint subsets of the set $\w_1\setminus \w$. Since $\I$ is a maximal ideal, either $M\setminus M_0\in \I$ or $f(M\setminus M_0)\in\I$. 

If $M\setminus M_0\in\I$ then $M\setminus M_0\in\mathcal B_X$ (as $\w\cap(M\setminus M_0)=\emptyset\in\J)$. If $f(M\setminus M_0)\in\I$, then $f(M\setminus M_0)\in\mathcal B_X$ (as $\w\cap f(M\setminus M_0)=\emptyset\in\J$) and then $M\setminus M_0\subset E^{-1}[f(M\setminus M_0)]\in\mathcal B_X$. 

In both cases we conclude that $M\setminus M_0\in\mathcal B_X$
 and hence $M_0\notin\mathcal B_X$ and $\w\cap M_0\notin\J$. It follows from $M_0\subset E^{-1}[f(M_0)]$ that $f(M_0)\notin\mathcal B_X$. Taking into account that $f(M_0)\in[X]^{\le\w}\subset \I$, we conclude that $\w\cap f(M_0)\notin \J$. So, $\w\cap M_0$ and $\w\cap f(M_0)$ are two disjoint subsets of $\w$ that do not belong to the ideal $\J$, which contradicts the maximality of the ideal $\J$.
\end{proof}


Finally, we present a simple example of a normal ballean which is not $\cof$-regular.

\begin{example}\label{ex:last} Let $\I$ be the ideal of sets $A\subset\omega$ such that $\lim_{n\to\infty}\frac{|A\cap[0,n)|}{n}=0$. On the set $X:=\omega$ consider the ball structure $\E_X:=\{E_{A,n}:A\in\I,\;n\in\omega\}$ consisting of entourages
$$E_{A,n}:=\Delta_X\cup\{(a,b)\in A\times \omega:|b-a|\le n\}\subset X\times X.$$
It is easy to see that bounded sets in the ballean $(X,\E_X)$ are finite and hence $\add(\mathcal E_X)=\add(\mathcal B_X)=\cof(\mathcal B_X)=\w$. On the other hand, a simple diagonal argument shows that $\cof(\mathcal E_X)>\w$, which means that the ballean $(X,\E_X)$ is not $\cof$-regular.

To see that the ballean $(X,\E_X)$ is normal, consider the normal ball structure $\mathcal M_X=\{\{(x,y)\in\w\times \w:|x-y|\le n\}:n\in\w\}$ induced by the Euclidean metric on $X=\w$. Now the normality of the ballean $(X,\E_X)$ follows from Proposition~\ref{p:emb} as the identity map from $(X,\E_X)$ to $(X,\mathcal M_X)$ is an asymptotic immersion (but not an asymorphism).
\end{example}



\section{Proof of Theorem~\ref{t:ultranorm}}\label{s:hyper1}

In the proof of Theorem~\ref{t:ultranorm} we shall use two lemmas.






\begin{lemma}\label{l:nik1} Let $X$ be a ballean. For any unbounded family $\A\subset [X]^{\mathcal B}\setminus[X]^{\le 1}$ in $[X]^{\mathcal B}$ there exists an unbounded set $V\subset X$ such that the family $\{A\in\A:A\cap V\ne A\}$ is unbounded in $[X]^{\mathcal B}$.
\end{lemma}

\begin{proof} Since the family $\A$ is unbounded in $[X]^{\mathcal B}$, its union $\bigcup\A$ is unbounded in $X$. Fix any well-order $\le$ on the set $\A$ and consider the subfamily $\tilde \A:=\{A\in\A:A\not\subset \bigcup_{B<A}B\}$. The family $\tilde \A$ can be written as $\tilde\A=\{A_\alpha\}_{\alpha\in\kappa}$ for some ordinal $\kappa$ such that $A_\alpha\not\subset\bigcup_{\beta<\alpha}A_\beta$ for all $\alpha\in\kappa$. Since the union $\bigcup_{\alpha<\kappa}A_\alpha=\bigcup\A$ in unbounded in $X$, the family $\{A_\alpha\}_{\alpha<\kappa}$ is unbounded in $[X]^{\mathcal B}$. 

By induction, for every $\alpha<\kappa$ we can choose a set $W_\alpha\subset A_\alpha\setminus\bigcup_{\beta<\alpha}A_\beta$ such that $\emptyset\ne \bigcup_{\beta\le \alpha}W_\beta\cap A_\alpha\ne A_\alpha$. The choice of $W_\alpha$ is possible as $|A_\alpha|\ge 2$.

Put $W:=\bigcup_{\alpha<\kappa}W_\alpha$ and observe that $\emptyset\ne W\cap A_\alpha\ne A_\alpha$ and $(X\setminus W)\cap A_\alpha\ne A_\alpha$ for every $\alpha<\kappa$. Since $X$ is unbounded, either $W$ or $X\setminus W$ is unbounded. If $W$ is unbounded, then put $V:=W$. Otherwise, put $V:=X\setminus W$. 

It follows that the family $\{A\in\A:A\cap V\ne A\}$ contains the unbounded family $\{A_\alpha\}_{\alpha<\kappa}$ and hence is unbounded in $[X]^{\mathcal B}$.
\end{proof} 

\begin{lemma}\label{l:nik2} Let $X$ be an ultradiscrete ballean and $n\in\mathbb N$. For any  unbounded family $\A\subset[X]^{\le n}$ there exists an unbounded set $V\subset X$ such that the family $\{A\in\A:|A\cap V|\le 1\}$ is unbounded in $[X]^{\mathcal B}$.
\end{lemma}

\begin{proof} If $n=1$, then the unbounded set $V=X$ has the required property.
Assume that the statement of the lemma has been proved for some $n\in\mathbb N$.
Take any unbounded family $\A\subset [X]^{\le n{+}1}$. If $\A\cap[X]^{\le 1}$ is unbounded, then the set $V=X$ has the required property: the family $\{A\in\A:|A\cap V|\le 1\}\supset\A\cap [X]^{\le 1}$ is unbounded.

So, we assume that the set $\A\cap [X]^{\le 1}$ is bounded and then the family $\A\setminus [X]^{\le 1}$ is unbounded. By Lemma~\ref{l:nik1}, there exists an unbounded set $U\subset X$ such that the family $\A_U:=\{A\in\A\setminus[X]^{\le 1}:A\cap U\ne A\}$ is unbounded in $[X]^{\mathcal B}$.

We claim that the family $\A'_U:=\{A\cap U:A\in \A_U\}\setminus\{\emptyset\}$ is unbounded in $[X]^{\mathcal B}$. Assuming that the family $\A'_U$ is bounded, we conclude that its union $\bigcup\A'_U$ is bounded in $X$. Since the ballean $X$ is ultradiscrete, the complement $X\setminus U$ of the unbounded set $U$ is bounded. Then the set $B=(\bigcup\A')\cup(X\setminus U)$ is bounded in $X$ and the family $\A_U\subset\{A\in[X]^{\mathcal B}:A\subset B\}$ in bounded in $[X]^{\mathcal B}$, which is a contradiction. This contradiction shows that the family $\A_U'$ is unbounded in $[X]^{\mathcal B}$. Since $\A_U'\subset [X]^{\le n}$, we can apply the inductive assumption and find an unbounded set $V\subset X$ such that the family $\{A\in\A'_U:|A\cap V|\le 1\}$ is unbounded in $[X]^{\mathcal B}$. Since the ballean $X$ is ultradiscrete, the intersection $U\cap V$ is unbounded. So we can replace $V$ by $V\cap U$ and assume that $V\subset U$. Then the family $\A'=\{A\in\A:|A\cap V|\le 1\}$ has unbounded union  $\bigcup\A'\supset\bigcup\A'_U$ and hence is unbounded in $[X]^{\mathcal B}$.
\end{proof}

Now we can present the {\em proof of Theorem~\ref{t:ultranorm}}. Given an ultradiscrete ballean $X$, we should prove that for every $n\ge 2$ the power $X^n$ is not normal but the hypersymmetric power $[X]^{\le n}$ is normal.

Theorem~\ref{t:main} and Proposition~\ref{p:ultra} imply that the square $X^2$ of $X$ is not normal. Since $X^2$ admits an asymorphic embedding into $X^n$, the ballean $X^n$ is not normal, too.
\smallskip

To show that the hypersymmetric power $[X]^{\le n}$ of $X$ is normal, fix any asymptotically disjoint unbounded sets $\A_1,\A_2\subset [X]^{\le n}$.

First we show that for some $i\in\{1,2\}$ the set $\A_i$ is {\em escaping} in the sense that for any bounded set $B\subset X$ the set $\{A\in\A_i:B\cap A\ne\emptyset\}$ is bounded in $[X]^{\mathcal B}$. 

To derive a contradiction, assume that none of the sets $\A_1,\A_2$ is escaping. In this case we can find a bounded set $B\subset X$ such that for every $i\in\{1,2\}$ the set $\A_i'=\{A\in\A_i:A\cap B\ne\emptyset\}$ is unbounded. By Lemma~\ref{l:nik2}, there exists an unbounded set $V_i\subset X$ such that the family $\A_i''=\{A\in\A_i':|A\cap V_i|\le 1\}$ is unbounded in $[X]^{\mathcal B}$. Since the ballean $X$ is ultradiscrete, the set $V=V_1\cap V_2\setminus B$ is unbounded and the set $X\setminus V\supset B$ is bounded in $X$.

Consider the entourage $E=\Delta_X\cup(X\setminus V)^2\in{\downarrow}\E_X$. Since the sets $\A_1$, $\A_2$ are asymptotically disjoint, the set $\hat E[\A_1]\cap\hat E[\A_2]$ is bounded in $[X]^{\mathcal B}$ and hence is contained in the family $\{A\in[X]^{\mathcal B}:A\subset D\}$ for some bounded set $D\subset X$ that contains $X\setminus V$. Since the ballean $(X,\E_X)$ is ultradiscrete, the unbounded sets $\bigcup\A_1''\setminus D$ and $\bigcup\A_2''\setminus D$ are not disjoint and hence contain some common point $x\notin D$, which belongs to some sets $A_1\in\A_1''$ and $A_2\in\A_2''$.  Fix any point $b\in B$. Since the sets $A_1$ and $A_2$ intersect the set $B\subset X\setminus V$, we obtain $\{x,b\}\in E[A_1]\cap E[A_2]$.

Taking into account that $X\setminus V\subset D$, we conclude that $x\in V$.
Since $x\in A_i\cap V$ and $|A_i\cap V|\le |A_i\cap V_i|\le1$, we have $A_i\subset \{x\}\cup (X\setminus V)=E[\{x,b\}]$. Then $\{x,b\}\in\hat E[A_1]\cap\hat E[A_2]$ and hence $\{x,b\}\subset D$, which contradicts the choice of the point $x$. This contradiction shows  that one of the sets $\A_1$ or $\A_2$ is escaping. 

We lose no generality assuming that the set $\A_1$ is escaping. In this case we shall prove that the sets ${\mathcal U}_1:=\A_1$ and ${\mathcal U}_2:=[X]^{\le n}\setminus \A_1$ are asymptotic neighborhoods of the sets $\A_1$ and $\A_2$ in $[X]^{\le n}$, respectively. To see that ${\mathcal U}_i$ is an asymptotic neighborhood of $\A_i$, take any entourage $E\in\E_X$ and find a bounded set $B\subset X$ such that $E[x]=E^{-1}[x]=\{x\}$ for all $x\in X\setminus B$. 

Since $\A_1$ is escaping, the set $\mathcal D:=\{A\in\A_1:A\cap B\ne\emptyset\}$ is bounded in $[X]^{\le n}$. It follows that for every $A\in\A_1\setminus\mathcal D$ we get $A\cap B=\emptyset$ and hence $\hat E[A]=\{A\}\subset \A_1={\mathcal U}_1$, which means that ${\mathcal U}_1$ is an asymptotic neighborhood of $\A_1$ in $[X]^{\le n}$.

Since the sets $\A_1,\A_2$ are asymptotically disjoint, the set $\hat E[\A_1]\cap\hat E[\A_2]$ is bounded. So, we can find a bounded set $D\subset X$ such that $\hat E[\A_1]\cap\hat E[\A_2]\subset\{A\in[X]^{\le n}:A\subset D\}$ and $B\subset D$. We claim that for any $A\in\A_2$ with $A\not\subset D$, we get $\hat E[A]\subset {\mathcal U}_2$. Otherwise, we could find a set $A'\in\hat E[A]\cap\A_1\subset\hat E[\A_2]\cap\hat E[\A_1]$ and hence $A'\subset D$ and $A\in E[A']\subset B\cup A'\subset D$, which contradicts the choice of $A$. This contradiction shows that ${\mathcal U}_2$ is an asymptotic neighborhood of $\A_2$.

\section{Proof of Theorem~\ref{t:hyper2}}\label{s:hyper2}

The proof of Theorem~\ref{t:hyper2} is divided into two lemmas.

\begin{lemma}\label{l:hyper1}  If for a ballean $X$ the hypersymmetric square $[X]^{\le 2}$ is normal, then either $X$ is ultranormal or the bornology $\mathcal B_X$ of $X$ has a linear base.
\end{lemma}

\begin{proof} Assume that a ballean $X$ is not ultranormal but its symmetric square $[X]^{\le 2}$ is normal. Since $X$ is not ultranormal, there exist two  (asymptotically) disjoint unbounded sets $A,B\subset X$. By Theorem~2.1 in  \cite{Prot}, there exists a function $\varphi:X\to[0,1]$ such that $\varphi(A)=\{0\}$, $\varphi(B)=\{1\}$ and $\varphi$ is {\em slowly oscillating} in the sense that for any $\e>0$ and $E\in\E_X$ there exists a bounded set $B\in\mathcal B_X$ such that $\mathrm{diam}\,\varphi(E[x])<\e$ for all $x\in X\setminus B$.

Two cases are possible.

1. For some non-empty open set $U\subset [0,1]$ the preimage $\varphi^{-1}(U)$ is bounded in $X$. In this case we can fix two real numbers $a<b$ in the interval $(0,1)$ such that $[a,b]\subset U$ and conclude that $Y:=\varphi^{-1}([0,a])$ and $Z:=\varphi^{-1}([b,1])$ are two asymptotically disjoint unbounded sets whose union $Y\cup Z$ has bounded complement $X\setminus(Y\cup Z)$.

The asymptotical disjointness of the sets $Y$ and $Z$ implies that the map $$Y\times Z\to [X]^{\le 2},\;\;(y,z)\mapsto \{y,z\},$$ is an asymorphic embedding. By Proposition~\ref{p:emb}, the normality of $[X]^{\le 2}$ implies the normality of the ballean $Y\times Z$. By Theorem~\ref{t:main},
$\add(\mathcal B_Y)=\cof(\mathcal B_Y)=\add(\mathcal B_Z)=\cof(\mathcal B_Z)$.
Taking into account that the complement $X\setminus(Y\cup Z)$ is bounded, we conclude that $$\add(\mathcal B_X)=\min\{\add(\mathcal B_Y),\add(\mathcal B_Z)\}=\max\{\cof(\mathcal B_Y),\cof(\mathcal B_Z)\}=\cof(\mathcal B_X),$$
which implies that the bornology $\mathcal B_X$ has a linear base.
\smallskip

2. For any non-empty open set $U\subset [0,1]$ the preimage $\varphi^{-1}(U)$ is unbounded in $X$. For every $i\in\{1,2,3,4\}$ consider the unbounded set $X_i=\varphi^{-1}([\tfrac{i-1}4,\tfrac{i}4])$ in $X$ and observe that for any numbers $i,j\in\{1,2,3,4\}$ with $|i-j|\ge 2$ the sets $X_i$ and $X_j$ are asymptotically disjoint. This fact can be used to show that the map $X_i\times X_j\to [X]^{\le 2}$,  $(x,y)\mapsto\{x,y\}$, is an asymorphic embedding.   By Proposition~\ref{p:emb}, the normality of $[X]^{\le 2}$ implies the normality of the ballean $X_i\times X_j$. By Theorem~\ref{t:main},
$$\add(\mathcal B_{X_i})=\cof(\mathcal B_{X_i})=\add(\mathcal B_{X_j})=\cof(\mathcal B_{X_j}).$$Consequently, there exists a cardinal $\kappa$ such that $\kappa=\add(\mathcal B_{X_i})=\cof(\mathcal B_{X_i})$ for all $i\in\{1,2,3,4\}$.
Taking into account that $X=\bigcup_{i=1}^4 X_i$, we conclude that
$$\add(\mathcal B_X)=\min_{1\le i\le 4}\add(\mathcal B_{X_i})=\kappa=\max_{1\le i\le 4}\cof(\mathcal B_{X_i})=\cof(\mathcal B_X),$$
which implies that the bornology $\mathcal B_X$ has a linear base.
\end{proof}

\begin{lemma}\label{p:hyper1} Assume that for a ballean $(X,\E_X)$ the symmetric square $[X]^{\le 2}$ is normal. Then
\begin{enumerate}
\item for every $x_0\in X$ there exists a monotone function $f:\E_X\to\mathcal B_X$ such that\newline $E_2[E_1[x_0]\setminus f(E_2)]\subset f(E_1)$ for any entourages $E_1,E_2\in\E_X$;
\item there exist monotone functions $f:\E_X\to\mathcal B_X$ and $g:\mathcal B_X\to\mathcal B_X$ such that $E[B\setminus f(E)]\subset g(B)$ for any $E\in\E_X$ and $B\in\mathcal B_X$;
\item the ballean $X$ has bounded growth.
\end{enumerate}
\end{lemma}

\begin{proof} 1. Fix any point $x_0\in X$ and consider two subsets $$D:=\big\{\{x_0,x\}:x\in X\big\}\mbox{ and }S:=\big\{\{x\}:x\in X\big\}$$in the symmetric square $[X]^{\le 2}$ of $X$. We claim that the sets $D$ and $S$ are asymptotically disjoint. Given any entourage $E\in \E_X$, we need to show that the intersection  $\hat E[D]\cap \hat E[S]$ is bounded in $[X]^{\le 2}$. Replacing $E$ by $E\cup E^{-1}$, we can assume that $E=E^{-1}$. Fix any elements $\{x_0,x\}\in D$ and $\{y\}\in S$ and take any set $A\in \hat E[\{x_0,x\}]\cap \hat E[\{y\}]$. It follows that $x_0\in E[A]$ and $A\subset E[y]$, which implies that $x_0\in E\circ E[y]$ and hence $y\in E^{-1}\circ E^{-1}[x_0]=E^2[x_0]$ and finally, $A\subset E[y]\subset E^3[x_0]$. Therefore, $\hat E[D]\cap \hat E[S]$ is contained in the bounded subset $\{A\in[X]^{\le 2}:A\subset E^3[x_0]\}$ of $[X]^{\le 2}$, witnessing that the sets $D$ and $S$ are asymptotically disjoint in $[X]^{\le 2}$. 

By the normality of $[X]^{\le 2}$, the asymptotically disjoint sets $D,S$ have disjoint asymptotic neighborhoods $O_D$ and $O_S$ in $[X]^{\le 2}$.
By the definition of an asymptotic neighborhood, for every $E\in\E_X$ the sets $$D_E:=\{A\in D:\hat E[A]\subset O_D\}\mbox{ and }S_E:=\{A\in S:\hat E[A]\subset O_S\}$$have bounded complements $D\setminus D_E$ and $S\setminus S_E$ in $[X]^{\le 2}$. Then the set 
$$f(E):=\{x\in X:\{x_0,x\}\in D\setminus D_E\}\cup\{y\in X:\{y\}\in S\setminus S_E\}$$ is bounded in $X$. It is easy to see that the function $f:\E_X\to\mathcal B_X$, $f:E\mapsto f(E)$, is monotone in the sense that $f(E)\subset f(E')$ for any entourages $E\subset E'$ in $\E_X$.

We claim that the function $f$ has the required property: $E_2[E_1[x_0]\setminus f(E_2)]\subset f(E_1)$ for any $E_1,E_2\in\E_X$.

To derive a contradiction, assume that  $E_2[E_1[x_0]\setminus f(E_2)]\not\subset f(E_1)$ for some $E_1,E_2\in\E_X$.
Then there exist points $y\in E_1[x_0]\setminus f(E_2)$ and $x\in E_2[y]\setminus f(E_1)$. It follows that $E_2[y]$ intersects the sets $E_1[x_0]$ and $E_1[x]$, which implies that $\hat E_2[\{y\}]\cap \hat E_1[\{x_0,x\}]\ne\emptyset$. On the other hand, $x\notin f(E_1)$ and $y\notin f(E_2)$ imply
that $\{x_0,x\}\in D_{E_1}$ and $\{y\}\in S_{E_2}$. Consequently,
$$\emptyset\ne \hat E_1[\{x_0,x\}]\cap \hat E_2[\{y\}]\subset O_D\cap O_S=\emptyset,$$
which is a desired contradiction.
\smallskip

2. Fix any point $x_0\in X$. By the preceding statement, there exists a monotone function $f:{\downarrow}\E_X\to\mathcal B_X$ such that 
$E[E_1[x_0]\setminus f(E)]\subset f(E_1)$ for any $E_1,E\in{\downarrow}\E_X$.

For every bounded set $B\in \mathcal B_X$, consider the entourage $$E_B:=\Delta_X\cup\big((B\cup\{x_0\})\times (B\cup\{x_0\})\big)\in{\downarrow}\E_X$$ and put $g(B):=f(E_B)$. The monotonicity of the function $f$ implies the monotonicity of the function $g:\mathcal B_X\to\mathcal B_X$, $g:B\mapsto g(B)=f(E_B)$. It remains to  observe that for any $E\in\E_X$ and $B\in\mathcal B_X$ we have
$$E[B\setminus f(E)]\subset E[E_B[x_0]\setminus f(E)]\subset  f(E_B)=g(B).$$ 
 \smallskip

3. By the preceding statement, there exist monotone functions  
$f:\E_X\to\mathcal B_X$ and $g:\mathcal B_X\to\mathcal B_x$ such that $E[B\setminus f(E)]\subset g(B)$ for any $E\in\E_X$ and $B\in\mathcal B_X$.
We claim that the entourage $G=\{(x,y)\in X\times X:y\in g(\{x\})\}$  witnesses that the ballean $X$ has bounded growth. First observe that for any bounded set $B\subset X$ the monotonicity of the function $g$ ensures that the set $G[B]\subset g(B)$ is bounded. Also, for any $E\in\E_X$ and $x\in X\setminus f(E)$ we have $E[x]= E[\{x\}\setminus f(E)]\subset g(\{x\})=G[x]$, which means that $G$ is a growth entourage for $X$.
\end{proof}

\section{Proof of Theorem~\ref{t:Gnorm1}}\label{s:Gnorm1}

Theorem~\ref{t:Gnorm1} follows from Proposition~\ref{p:image} and the following lemma.

\begin{lemma} Let $n\in\IN$ and $G\subset H$ be two subgroups of the symmetric group $S_n$. For every ballean $X$ the map $\pi:[X]^n_G\to [X]^n_H$, $\pi:xG\mapsto xH$, is perfect.
\end{lemma}

\begin{proof} The definition of the ball structures on the balleans $[X]^n_G$ and $[X]^n_H$ implies that the map $\pi:[X]^n_G\to [X]^n_H$, $\pi:xG\mapsto xH$, is proper and macro-uniform. To see that this map is closed, fix any asymptotically disjoint sets $A,B\subset [X]^n_G$ such that $A=\pi^{-1}(\pi(A))$. We need to check that the sets $\pi(A)$ and $\pi(B)$ are asymptotically disjoint in $[X]^n_H$. Fix any entourage $E=E^{-1}\in \E_X$ and consider the entourage $$\tilde E=\{(x,y)\in X^n\times X^n:\forall i\in n\;\;(x(i),y(i))\in  E\}\in\E_{X^n}.$$ The entourage $\tilde E$ induces the entourages $\hat E_G=\{(xG,yG):(x,y)\in \tilde E\}\in\E_{[X]^n_G}$ and $\hat E_H=\{(xH,yH):(x,y)\in \tilde E\}\in\E_{[X]^n_H}$. Since the set $A,B$ are asymptotically disjoint in $[X]^n_G$, the intersection $\hat E_G[A]\cap\hat E_G[B]$ is bounded in $[X]^n_G$ and hence is contained in the set $\{xG:x\in D^n\}$ for some bounded set $D\subset X$.

We claim that $\hat E_H[\pi(A)]\cap\hat E_H[\pi(B)]\subset \{xH:x\in D^n\}$. Fix any elements $aG\in A$ and $bG\in B$ and $cH\in \hat E_H[\pi(aG)]\cap\hat E_H[\pi(bG)]=
\hat E_H[aH]\cap\hat E_H[bH]$. Taking into account that $(bH,cH)\in\hat E_H$, we can replace $c\in X^n$ by a suitable representative in the equivalence classe $cH$ and assume that $(b,c)\in \tilde E$ and hence $cG\in \hat E_G[bG]$. On the other hand, $cH\in\hat E_H[aH]$ implies that $c\in \tilde E[a\circ h]$ for some $h\in H$.
Now $(a\circ h)H=aH=\pi(aG)\in\pi(A)$ and the equality $A=\pi^{-1}(\pi(A))$ imply  $(a\circ h)G\in\pi^{-1}(\pi(A))=A$ and $cG\in\hat E_G[A]\cap \hat E_G[bG]\subset \hat E_G[A]\cap \hat E_G[B]\subset \{xG:x\in D^n\}$.
Then $cH\in\{xH:x\in D^n\}$ and hence  $\hat E_H[\pi(A)]\cap\hat E_H[\pi(B)]\subset \{xH:x\in D^n\}$.
\end{proof} 

\section{Proof of Theorem~\ref{t:Gnorm2}}\label{s:Gnorm2}

Theorem~\ref{t:Gnorm2} follows from Lemmas~\ref{l:Gnorm1} and \ref{l:Gnorm2}, proved in this section.

Let $X$ be a ballean and $z\in X$ be any point. For $k<n$ let
$$\Delta_k^n(z):=\{x\in X^n:|x[n]\setminus\{z\}|=1,\;|x^{-1}(z)|=k\},$$
where $x[n]=\{x(i):i\in n\}$.

\begin{lemma}\label{l:Gnorm0} The bornology of the ballean $X$ has a linealy ordered base if for some positive $k<l<n$ the sets $\Delta^n_k(z)$ and $\Delta^n_l(z)$ have disjoint asymptotic neighborhoods in $X^n$.
\end{lemma}

\begin{proof} Assume that $U,V$ are two disjoint asymptotic neighborhoods of the sets $\Delta^n_k(z)$ and $\Delta^n_l(z)$ in $X^n$, respectively. Then for any entourage $E\in\E_X$ there exists a bounded set $\varphi(E)\subset X$ such that $\hat E[\Delta^n_k(z)\setminus \varphi(E)^n]\subset U$ and $\hat E[\Delta^n_l(z)\setminus \varphi(E)^n]\subset V$ where $\hat E=\{(x,y)\in X^n\times X^n:\forall i\in n\;\;(x(i),y(i))\in E\}\in\E_{X^n}$.

For every bounded set $B\in\mathcal B_X$ find an entourage $E_B=E_B^{-1}\in\E_X$ such that $B\subset E_B[z]$. We claim that the function $f:\mathcal B_X\to\mathcal B_X$,  $f:B\mapsto \varphi(E_B)$, has the property: for any bounded sets $B,D$ in $X$ we have $B\subset f(D)$ or $D\subset f(B)$.

To derive a contradiction, assume that there are two bounded sets $B,D\subset X$ such that $B\not\subset f(D)=\varphi(E_D)$ and $D\not\subset f(B)=\varphi(E_B)$.
Choose elements $x\in \Delta^n_k(z)$ and $y\in \Delta^n_l(z)$ such that $x[n]\subset \{z\}\cup (B\setminus\varphi(E_D))$ and $y[n]\subset \{z\}\cup(D\setminus\varphi(E_B))$. Then $\hat E_D[x]\subset U$ and $\hat E_B[y]\subset V$.

On the other hand, $x[n]\subset \{z\}\cup B\subset E_B[z]$ and $y[n]\subset\{z\}\cup D\subset E_D[z]$. Consequently, $z\in E^{-1}_B[x(i)]\cap E^{-1}_D[x(i)]=E_B[x(i)]\cap E_D[x(i)]$ for all $i\in n$ and the constant function $\bar z:n\to\{z\}$ belongs to $\hat E_B[x]\cap\hat E_D[y]\subset U\cap V=\emptyset$. This is a desired contradiction showing that  for any bounded sets $B,D$ in $X$ we have $B\subset f(D)$ or $D\subset f(B)$. Now we can apply Lemma~\ref{l2} and conclude that $\add(\mathcal B_X)=\cof(\mathcal B_X)$, which implies that the bornology $\mathcal B_X$ of $X$ has a linearly ordered base.
\end{proof}

\begin{lemma}\label{l:Gnorm1} The bornology of a ballean $X$ has a linearly ordered base if for some $n\ge 3$ and some subgroup $G\subset S_n$ the $G$-symmetric power $[X]^{n}_G$ is normal.
\end{lemma}  

\begin{proof} Assume that for some $n\ge 3$ and some subgroup $G\subset S_n$ the ballean $[X]^n_G$ is normal. Let $\pi:X^n\to [X]^n_G$ be the surjective proper macro-uniform map assigning to each function $x\in X^n$ its equivalence class $xG=\{x\circ g:g\in G\}$. Since $n\ge 3$, we can fix two positive integer numbers $k<l<n$.

We claim that the sets $\pi(\Delta^n_k(z))$ and $\pi(\Delta^n_l(z))$ are asymptotically disjoint in $[X]^n_G$.

 Given any entourage $E\in\E_X$ consider the entourages $$\tilde E=\{(x,y)\in X^n\times X^n:\forall i\in n\;(x(i),y(i))\in E\}\in\E_{X^n}$$and 
 $$\hat E=\{(xG,yG):(x,y)\in\tilde E\}\in\E_{[X]^n_G}.$$
 
 We need to prove that the intersection $\hat E[\pi(\Delta^n_l(z))]\cap\hat E[\pi(\Delta^n_l(z))]$ is bounded in $[X]^n_G$.
 
Take any element $xG\in\hat E[\pi(\Delta^n_k(z))]\cap \hat E[\pi(\Delta^n_l(z))]$.
Taking into account that the sets $\Delta^n_k(z)$ and $\Delta^n_l(z)$ are invariant under the (coordinate permutating) action of the group $G$, we conclude that $x\in \tilde E[a]\cap \hat E[b]$ for some $a\in\Delta^n_k(z)$ and $b\in\Delta^n_l(z)$. By the definition of the sets $\Delta^n_k(z)$ and $\Delta^n_l(z)$, the sets $a[n]\setminus\{z\}$ and $b[n]\setminus z$ are singletons. Since $|a^{-1}(z)|=k<l=|b^{-1}(z)|$, there exists $i\in n$ such that $b(i)=z\ne a(i)$. It follows that $x(i)\in E[a(i)]\cap E[b(i)]=E[a(i)]\cap E[z]$ and thus $a(i)\in E^{-1}{\circ}E[z]$ and $a[n]=\{z,a(i)\}\subset E^{-1}{\circ}E[z]$.
Then the set $x[n]\subset E[a[n]]\subset E{\circ}E^{-1}{\circ}E[z]$ is bounded in $X$ and hence the intersection $$\hat E[\pi(\Delta^n_k(z))]\cap \hat E[\pi(\Delta^n_l(z))]\subset \{xG\in [X]^n_G:x[n]\subset E{\circ}E^{-1}{\circ}E[z]\}$$ is bounded in $[X]_G^n$.

By the normality of the ballean $[X]^n_G$, the asymptotically disjoint sets $\pi(\Delta^n_k(z))$ and $\pi(\Delta^n_l(z))$ have disjoint asymptotic neighborhoods $U,V\subset [X]^n_G$. Taking into account that the map $\pi:X^n\to [X]^n_G$ is macro-uniform and proper, we can check that $\pi^{-1}(U)$, $\pi^{-1}(V)$ are disjoint asymptotic neighborhoods of the sets $\Delta^n_k(z)$ and $\Delta^n_l(z)$ in $X^n$. By Lemma~\ref{l:Gnorm0}, the bornology $\mathcal B_X$ of $X$ has a linearly ordered base.
\end{proof} 

\begin{lemma}\label{l:Gnorm2} A ballean $X$ has bounded growth if for some $n\ge 2$ and some subgroup $G\subset S_n$ the ballean $[X]^n_G$ is normal.
\end{lemma}

\begin{proof} By Theorem~\ref{t:Gnorm1}, the normality of $[X]^n_G$ implies the normality of the symmetric $n$-th power $[X]^n$ of $X$. Without loss of generality, we can assume that $X$ is not empty and hence contains some point $z\in X$. It can be shown that the map $f:[X]^{\le 2}\to [X]^n$, $f:\{x,y\}\mapsto (x,y,z,\dots,z)$, is an asymorphic embedding. By Proposition~\ref{p:emb}, the normality of the ballean $[X]^n$ implies the normality of $[X]^{\le 2}$. Applying Lemma~\ref{p:hyper1}(3), we conclude that the ballean $X$ has bounded growth.
\end{proof}

\section{Balleans on $G$-spaces}\label{s:Gspace}

In this section we study the finitary ball structure on transitive $G$-spaces, i.e., sets $X$ endowed with a transitive (left) action of a group $G$. The action of $G$ on $X$ is {\em transitive} if $Gx=X$ for all $x\in X$.
 
Each transitive $G$-space $X$ carries the canonical ball structure $\E_{X,G}$ consisting of the entourages $$E_F:=\big\{(x,y)\in X\times X:y\in\{x\}\cup\{gx\}_{g\in F}\big\}$$ 
where $F$ runs over finite subsets of the group $G$. 

The ballean $(X,\E_{X,G})$ is {\em finitary} in the sense that $\sup\{|E[x]|:x\in X\}<\infty$ for every entourage $E\in\E_{X,G}$.

A subgroup $G$ of the permutation group $S_X$ of a set $X$ is called {\em transitive} if its action on $X$ is {\em transitive}, which means that $Gx=X$ for all $x\in X$.
By Theorem 1 proved in \cite{P2008}, for each finitary ballean $(X,\E_X)$ there exists a transitive subgroup $G\subset S_X$ such that the balleans $\E_X$ and $\E_{X,G}$ generate the same coarse structure, i.e., ${\downarrow}\E_X={\downarrow}\E_{X,G}$. So, the study of finitary balleans can be reduced to investigation of the canonical balleans on transitive $G$-spaces.

In this respect we can ask the following natural problem.

\begin{problem} Study the interplay between algebraic and topological properties of a transitive subgroup $G\subset S_X$ and asymptotic properties of the ballean $(X,\E_{X,G})$? 
\end{problem}

Here we endow the permutation group $S_X$ with the topology inherited from the topology of the Tychonoff product $X^X$, where $X$ is endowed with the discrete topology. This topology turns $S_X$ into a complete topological group. The topological group $S_X$ is Polish if the set $X$ is countable.

The bounded growth of transitive $G$-spaces has the following topological characterization.  

\begin{proposition} For a countable set $X$ and a transitive subgroup $G\subset S_X$ the ballean $(X,\E_{X,G})$ has bounded growth if and only if the group $G$ is contained in a $\sigma$-compact subset of $S_X$.
\end{proposition}

\begin{proof} Since the ballean $(X,\E_{X,G})$ is finitary, its bornology $\mathcal B_X$ coincides with the family $[X]^{<\w}$ of all finite subsets of $X$. 

If the ballean $(X,\E_{X,G})$ has bounded growth, then there exists a growth entourage $\Gamma$ for $X$. 

Fix any countable dense set $D\subset G$. For every $d\in D$ and $n\in\w$ consider the compact set 
$$K_{d,n}:=\{g\in S_X:(\forall k\le n\; g(x_k)=d(x_k))\mbox{ and }
(\forall k>n\;\{g(x_k),g^{-1}(x_k)\}\in\Gamma[x_k])\}$$in the Polish group $S_X$. 
The choice of the growth function $\gamma$ guarantees that $$G\subset \bigcup_{d\in D}\bigcup_{n\in\w}K_{d,n}.$$ 
\smallskip

Now assume that the subgroup $G$ is contained in a $\sigma$-compact set $A\subset S_X$. Write $A$ as the countable union $A=\bigcup_{n\in\w}A_n$ of compact sets $A_n$ such that $A_n\subset A_{n+1}$ for all $n\in\w$. It follows that for every $n\in\w$ the set $\Gamma[x_n]:=\{x_n\}\cup\{f(x_n):f\in A_n\}\subset X$ is finite.
We claim that the entourage $\Gamma=\bigcup_{x\in X}\Gamma[x]$ witnesses that the ballean $(X,\E_{X,G})$ has bounded growth. Indeed, for any finite set $F\subset G\subset \bigcup_{n\in\w}A_n$, we can find $n\in\w$ such that $F\subset A_n$. Then for every $m\ge n$ we have
$$E_F[x_m]=\{x_m\}\cup\{f(x_m):f\in F\}\subset \{x_m\}\cup\{f(x_m):f\in A_m\}=\Gamma[x_m].$$
\end{proof}

A sufficient condition for the (ultra)normality of $(X,\E_{X,G})$ is the infinite mixing property of the action of $G$ on $X$.

\begin{definition} We say that a subgroup $G\subset S_X$ is {\em infinitely mixing} if for any infinite sets $I,J\subset X$ there exists a permutation $g\in G$ such that the intersection $g(I)\cap J$ is infinite.
\end{definition}

\begin{proposition}\label{p:mix} If a transitive subgroup $G\subset S_X$ is infinitely mixing, then the ballean $(X,\E_{X,G})$ contains no unbounded asymptotically disjoint subsets and hence is ultranormal.
\end{proposition}

\begin{proof} Given two unbounded (and hence infinite) asymptotically disjoint sets $A,B\subset X$, we can use the infinitely mixing property and find a permutation $g\in G$ such that the intersection $g(A)\cap B$ is infinite. Then for the entourage $E_{\{g\}}:=\{(x,y)\in X\times X:y\in\{x,g(x)\}\}$ the intersection $E_{\{g\}}[A]\cap E_{\{g\}}[B]\supset g(A)\cap B$ is infinite and hence unbounded in $X$. This means that the sets $A,B$ are not asymptotically disjoint.
\end{proof}

Next, we find a condition on the action of the group $G$ guaranteeing that the ballean $(X,\E_{X,G})$ is pseudobounded.

Let us recall that a function $\varphi :X\to\IR$ of a ballean $X$ is {\em slowly oscillating} if for every $\e>0$ and every $E\in\E_X$ there exists a bounded set $B\subset X$ such that for every $x\in X\setminus B$ the set $\varphi(E[x])$ has diameter $<\e$ in the real line. 

A ballean $X$ is called {\em pseudobounded} if for each slowly oscillating function $\varphi :X\to\IR$ there exists a bounded set $B\subset X$ such that the set $\varphi(X\setminus B)$ is bounded in the real line. Pseudobounded balleans were introduced in \cite{Prot}. 


\begin{lemma} A ballean $X$ is pseudobounded if $\cof(\mathcal B_X)=\w$ and $X$ contains no discrete subballeans.
\end{lemma}

\begin{proof} Since  $\cof(\mathcal B_X)=\w$, we can choose a well-ordered base $(B_n)_{n\in\w}$ of the bornology of $X$.

Assuming that $X$ is not pseudobounded, we can find a slowly oscillating function $\varphi:X\to\IR$ such that for every $n\in\w$ the set $\varphi(X\setminus B_n)$ is unbounded in the real line.

Let $x_0\in X$ be any point. By induction for every $n\in\w$ we can select a point $X\setminus B_n$ such that $|\varphi(x_{n+1})|>|\varphi(x_n)|+1$. for every $n\in\w$. It can be shown that the subballean $\{x_n\}_{n\in\w}$ is unbounded and y discrete.
\end{proof}

The following proposition can be easily derived from the definitions of the ball structure $\E_{X,G}$ and Proposition~\ref{p:discrete}.
 
\begin{proposition}\label{p:pseudo} Let $X$ be a countable set and $G\subset S_X$ be a subgroup such that for any infinite subset $I\subset X$ there exists a permutation $g\in G$ such that the set $\{x\in I:x\ne g(x)\in I\}$  is infinite. Then the ballean $(X,\E_{X,G})$ contains no discrete subballeans and hence is pseudobounded but fails to have bounded growth.
\end{proposition}

Propositions~\ref{p:mix} and \ref{p:pseudo} imply 

\begin{corollary} For a countable set $X$ the ballean $(X,\E_{X,S_X})$ is ultranormal and pseudobounded but does not contain discrete subballeans and fails to have bounded growth.
\end{corollary} 

Finally let us prove that the normality of ball structures of $G$-spaces is preserved by equivariant maps. A map $\varphi:X\to Y$ between $G$-spaces is called {\em equivariant} if $\varphi(g{\cdot}x)=g{\cdot}\varphi(x)$ for all $g\in G$ and $x\in X$.

\begin{theorem}\label{t:equi} Let $\varphi:X\to Y$ be an equivariant map between two transitive $G$-spaces. If the ballean $(X,\E_{X,G})$ is normal, then so is the ballean $(Y,\E_{Y,G})$.
\end{theorem}

\begin{proof} The equivariance of $\varphi$ implies that the map $\varphi:(X,\E_{X,G})\to (Y,\E_{Y,G})$ is macro-uniform and open. The transitivity of the $G$-spaces implies that the equivariant map $\varphi$ is surjective. Since the ballean $(Y,\E_{Y,G})$ is finitary, any section $s:Y\to X$ of the map $\varphi$ is bornologous. Now we can apply Proposition~\ref{p:open} and conclude that the normality of the ballean $(X,\E_{X,G})$ implies the normality of the ballean $(Y,\E_{Y,G})$.
\end{proof}

\section{The proof of Theorem~\ref{t:group}}\label{s:group}

Given a group $G$ and an infinite cardinal $\kappa<|G|$, we should prove that
the ballean $(G,\E_{[G]^{<\kappa}})$ is not normal.
To derive a contradiction, assume that the ballean $(G,\E_{[G]^{<\kappa}})$ is normal.

Two cases are possible.

I. There exists a subgroup $A\subset G$ of cardinality $|A|=\kappa$ and a subset $M\subset G$ of cardinality $|M|>\kappa$ such that for each point $b\in M$ there exists a set $F\in[G]^{<\kappa}$ such that $|Ab^{-1}\cap FA|=\kappa$. By the Kuratowski-Zorn Lemma, the set $M$ contains a maximal subset $B$ such that the family $(bA)_{b\in B}$ is disjoint. It is easy to see that $|B|=|M|>\kappa$.

We claim that the sets $A$ and $B$ are asymptotically disjoint in the ballean $(G,\E_{[G]^{<\kappa}})$. 
First we show that for any $x,y\in G$ the intersection $xA\cap yB$ contains at most one point. To derive a contradiction, assume that $xA\cap yB$ contains two distinct points $z_1,z_2$. For every $i\in\{1,2\}$ find points $a_i\in A$ and $b_i\in B$ such that $xa_i=z_i=yb_i$. Then $b_1a_1^{-1}=y^{-1}x=b_2a_2^{-1}$ and hence $b_1=b_2$ (as the family $(bA)_{b\in B}$ is disjoint). Then $z_1=yb_1=yb_2=z_2$, which contradicts the choice of the points $z_1\ne z_2$. Then for any set $F\subset G$ of cardinality $|F|<\kappa$, we have $|FA\cap FB|\le\sum_{x,y\in F}|xA\cap yB|\le |F\times F|<\kappa$, which means that the sets $A,B$ are asymptotically disjoint.

By the normality of the ballean $(G,\E_{[G]^{<\kappa}})$, the asymptotically disjoint sets $A,B$ have disjoint asymptotic neighborhoods $O_A,O_B$. Then for every set $F\in[G]^{<\kappa}$ the set $\Phi(F):=\{a\in A:Fa\not\subset O_A\}\cup\{b\in B:Fb\not\subset O_B\}$ is bounded and hence has cardinality $|\Phi(F)|<\kappa$.
Consequently, the union $\Phi[A]=\bigcup_{a\in A}\Phi(\{a\})$ has cardinality $|\Phi[A]|\le|A|\cdot\kappa=\kappa$.

We claim that there are elements $x,y\in G$ and $a\in A\setminus\Phi(x)$ and $b\in B\setminus \Phi(y)$ such that $xa=yb$.
Since $|B|>\kappa\ge |\Phi[A]|$, there exists an element $b\in B\setminus\Phi[A]$. By our assumption, for the point $b\in B\subset M$ there exists a set $F_b\in[G]^{<\kappa}$ such that $|Ab^{-1}\cap F_bA|=\kappa$.

Since the set $\Phi(F_b^{-1})$ has cardinality $|\Phi(F_b^{-1})|<\kappa=|Ab^{-1}\cap F_bA|$, there exist  elements $a\in A$ and $f\in F_b$ such that $ab^{-1}\in fA$ and $a\notin\Phi(F_b^{-1})$. Put $x=f^{-1}$ and $y=f^{-1}ab^{-1}\in A$. Observe that $a\notin\Phi(F_b^{-1})\supset\Phi(\{x\})$ and $b\notin\Phi(\{y\})\subset \Phi[A]$. Consequently, $xa\in O_A$ and $yb\in O_B$. On the other hand, $xa=f^{-1}a=f^{-1}ab^{-1}b=yb\in O_A\cap O_B=\emptyset$, which is a desired contradiction completing the proof in the case I.
\smallskip

II. There exists a subgroup $A\subset G$ of cardinality $|A|=\kappa$ and a subset $M\subset G$ of cardinality $|M|>\kappa$ such that for every $b\in M$ and set $F\in[G]^{<\kappa}$ we have $|Ab\cap FA|<\kappa$. Using the Kuratowski-Zorn Lemma, choose a maximal subset $U\subset M$ such that  the family $(AbA)_{b\in U}$ is disjoint. By the maximality, the set $U$ has cardinality $|U|=|M|>\kappa$. Using the Kuratowski-Zorn Lemma, for every $b\in U$ choose a maximal set $A_b\subset A$ such that the family $(abA)_{a\in A_b}$ is disjoint. The maximality of $A_b$ implies that $AbA=A_bbA$. We claim that $|A_b|=\kappa$. In the opposite case, the set $F=A_bb$ has cardinality $|F|<\kappa$ and then the set $Ab=FA\cap Ab$ has cardinality $<\kappa$, which is not possible as $|Ab|=|A|=\kappa$. This contradiction shows that $|A_b|=\kappa$.

We claim that the sets $A$ and $B:=\{aba^{-1}:a\in A_b,\;b\in U\}$ are asymptotically disjoint in the ballean $(G,\E_{[G]^{<\kappa}})$. This will follow as soon as we check that for every $x\in G$ the set $xA\cap B$ contains at most one point. Assuming that $xA\cap B$ is not empty, find a point $b\in U$ such that $xA\cap AbA\ne\emptyset$. Then $x\in AbA$ and hence $xA\cap B\subset \bigcup_{u\in U}(AbA\cap AuA)=AbA$ as the family $(AuA)_{u\in U}$ is disjoint. Therefore, $xA\cap B=xA\cap B\cap AbA=xA\cap\{aba^{-1}:a\in A_b\}$. Assuming that  the set $xA\cap B$ contains two distinct points, we would find two distinct points $a_1,a_2\in A_b$ such that $a_1ba_1^{-1},a_2ba_2^{-1}\in xA$ and hence $x\in a_1bA\cap a_2bA$, which is not possible as the family $(abA)_{a\in A_b}$ is disjoint. This contradiction shows that $|xA\cap B|\le 1$ for every $x\in G$ and hence $|FA\cap FB|\le|F\times F|<\kappa$ for every $F\in[G]^{<\kappa}$.

By the normality of the ballean $(G,\E_{[G]^{<\kappa}})$, the asymptotically disjoint sets $A,B$ have disjoint asymptotic neighborhoods $O_A,O_B$. Then for every $x\in G$ the set $\Phi(x):=\{a\in A:xa\notin O_A\}\cup\{b\in B:xb\notin O_B\}$ is bounded and hence has cardinality $|\Phi(x)|<\kappa$.
  

Since $|U|>\kappa$, there exists an element $u\in U\setminus \bigcup_{a\in A}a\Phi(a)a^{-1}$. Since $|A_b|=\kappa>|\Phi(u)|$, there exists an element $a\in A_b\setminus \Phi(u)$. Put $y=a\in A$, $x=u$ and $b=a^{-1}ua\in B$. It follows that $a\notin\Phi(u)=\Phi(x)$ and $b=a^{-1}ua\notin \Phi(a)$ (as $u\notin a\Phi(a)a^{-1}$), which implies $xa\in O_A$ and $yb\in O_B$.
On the other hand, $xa=ua=aa^{-1}ua=yb\in O_A\cap O_B=\emptyset$.
This contradiction completes the proof of the case II and also the proof of Theorem~\ref{t:group}.

\section{The proof of Theorem~\ref{t:group2}}\label{s:group2}

Given any group $G$ and an infinite cardinal $\kappa\le|G|$
we should prove the equivalence of the following conditions:
\begin{enumerate}
\item the bornology $[G]^{<\kappa}$ of the ballean $(G,\E_{[G]^{<\kappa}})$ has a linearly ordered base;
\item the ball structure $\E_{[G]^{<\kappa}}$ has a linearly ordered base;
\item the ballean $(G,\E_{[G]^{<\kappa}})$ has bounded growth;
\item $|G|=\kappa$ and $\kappa$ is a regular cardinal.
\end{enumerate}
Moreover, if $\kappa$ os regular or $G$ is solvable, then we should prove that the conditions (1)--(4) are equivalent to the condition
\begin{itemize}
\item[(5)] The ballean $(G,\E_{[G]^{<\kappa}})$ is normal.
\end{itemize}

The implication $(1)\Ra(2)$ follows from the definition of the ball structure $\E_{[G]^{<\kappa}}$ and $(2)\Ra(3)$ from Proposition~\ref{p:lg}.
\smallskip

$(3)\Ra(4)$: Assume that the ballean $(G,\E_{[G]^{<\kappa}})$ has bounded growth, which means that there  exists an entourage $\Gamma$ on $X$  such that for any set $I\in[G]^{<\kappa}$ the sets $\Gamma[I]$ and $\{x\in G:\{x\}\cup Ix\not\subset \Gamma[x]\}$ belong to the ideal $[G]^{<\kappa}$.
Then the family $\{\Gamma(x)x^{-1}\}_{x\in G}$ is cofinal in the poset $[G]^{<\kappa}$, which means that $\cof([G]^{<\kappa})\le|G|$.

We claim that $|G|=\kappa$. In the opposite case we can take any subset $K\subset G$ of cardinality $|K|=\kappa<|G|$ and choose an element $g\in G\setminus\bigcup_{x\in K}\Gamma[x]x^{-1}$. Then for the singleton $\{g\}\in[G]^{<\kappa}$ the set $\{x\in G:\{x,gx\}\notin\Gamma[x]\}$ contains the set $K$ and hence does not belong to the ideal $[G]^{<\kappa}$. But this contradicts the choice of $\Gamma$. 

Now we see that $\cof([\kappa]^{<\kappa})=\cof([G]^{<\kappa})\le|G|=\kappa$.
It remains to prove that the cardinal $\kappa$ is regular. Since $\cof([\kappa]^{<\kappa})\le\kappa$, the poset $[\kappa]^{<\kappa}$ has a cofinal set $\{S_\alpha\}_{\alpha\in\kappa}$ of cardinality $\kappa$. Let $C\subset\kappa$ be a cofinal set of cardinality $|C|=\cof(\kappa)$. For every $\alpha\in C$ consider the set $U_\alpha=\bigcup\{S_\beta:\beta\le\alpha,\;\;|S_\beta|\le|\alpha|\}$ and observe that it has cardinality $|U_\alpha|<\kappa$. So, we can choose a point $x_\alpha\in\kappa\setminus U_\alpha$. Assuming that the cardinal $\kappa$ is singular, we conclude that the set $X=\{x_\alpha\}_{\alpha\in C}$ has cardinality $|X|\le\cof(\kappa)<\kappa$ and hence is contained in some set $S_\alpha$, where $\alpha\in\kappa$. Then for the ordinal $\beta:=\alpha+|S_\alpha|<\kappa$ we get $X\subset S_\alpha\subset U_\beta$, which is not possible as $x_\beta\in X\setminus U_\beta$. This contradiction shows that $\cof(\kappa)=\kappa$ and the cardinal $\kappa$ is regular.
\smallskip

$(4)\Ra(1)$: If $|G|=\kappa$ and the cardinal $\kappa$ is regular, then for any bijective function $f:\kappa\to G$  the family $\{f([0,\alpha])\}_{\alpha<\kappa}$ a linearly ordered base of the bornology $[G]^{<\kappa}$.
\smallskip

The implication $(2)\Ra(5)$ follows from Theorem~\ref{t:prot}.
\smallskip

Now assuming that the cardinal $\kappa$ is regular or the group $G$ is solvable, we shall prove that $(5)\Ra(4)$. If the cardinal $\kappa$ is regular, then the implication $(5)\Ra(4)$ follows from Theorem~\ref{t:group}. For solvable groups this implication is proved in the last statement of the following lemma.

\begin{lemma} Let $G$ be a group of singular cardinality $\kappa=|G|$.
If the ballean $(G,\E_{[G]^{<\kappa}})$ is normal, then
\begin{enumerate}
\item for each subset $A\subset G$ of cardinality $|A|=\kappa$ the centralizer\newline $\mathsf C_G(A)=\bigcap_{a\in A}\{x\in G:xa=ax\}$ has cardinality $|\mathsf C_G(A)|<\kappa$;
\item for any quotient group $H$ of $G$ the ballean $(H,\E_{[H]^{<\kappa}})$ is normal;
\item no subgroup of $G$ admits a homomorphism onto a group $H$ containing a subset $A\subset H$ with $|A|=\kappa=|\mathsf C_H(A)|$.
\item $G$ is not solvable.
\end{enumerate}
\end{lemma}

\begin{proof} 1. To derive a contradiction, assume that $G$ contains a subset $\Lambda\subset G$ such that $|\Lambda|=\kappa=|\mathsf C_G(\Lambda)|$. By transfinite induction we can construct two transfinite sequences $\{a_\alpha\}_{\alpha\in\kappa}\subset \Lambda$ and $\{b_\alpha\}_{\alpha\in\kappa}\subset \mathsf C_G(\Lambda)$ such that for every $\alpha<\kappa$ the following conditions hold:
\begin{itemize}
\item $a_\alpha\ne a_ib_j^{-1}b_k$ for any ordinals $i,j,k<\alpha$;
\item $b_\alpha\in b_ia^{-1}_ja_k$ for any ordinals $i<\alpha$ and $j,k\le\alpha$.
\end{itemize}
The choice of the elements $a_\alpha,b_\alpha$ is always possible since $|\Lambda|=\kappa=|\mathsf C_G(\Lambda)|$. Then for the sets $A=\{a_\alpha\}_{\alpha\in\kappa}$ and   $B=\{b_\alpha\}_{\alpha\in\kappa}$ the intersection $A^{-1}A\cap B^{-1}B$ consists the unique element equal to the unit $e$ of the group $G$. This property can be used to show that for every $x,y\in G$ the intersection $xA\cap yB$ is a singleton. This implies that for any subset $F\in[G]^{<\kappa}$ the intersection $FA\cap FB$ belongs to the ideal $[G]^{<\kappa}$, which means that the sets $A,B$ are asymptotically disjoint in the ballean $(G,\E_{[G]^{<\kappa}})$.

By the normality of the ballean $(G,\E_{[G]^{<\kappa}})$, the asymptotically disjoint sets $A,B$ has disjoint asymptotic neighborhoods $O_A$ and $O_B$.
Then the functions 
$$\varphi:[A]^{<\kappa}\to[B]^{<\kappa},\;\;\varphi:F\mapsto\{b\in B:Fb\not\subset O_B\}$$and
$$\psi:[B]^{<\kappa}\to[A]^{<\kappa},\;\;\psi:F\mapsto\{a\in A:Fa\not\subset O_A\}$$are well-defined. We claim that for any $A'\in[A]^{<\kappa}$ and $B'\in[B]^{<\kappa}$ we have $B'\subset\varphi(A')$ or $A'\subset\psi(B')$. Assuming that 
$B'\not\subset\varphi(A')$ and $A'\not\subset\psi(B')$, we can find points $a\in A'\setminus\psi(B')$ and $b\in B'\setminus\varphi(A')$ and conclude that $ab\subset O_B$ and $ba\in O_A$. Taking into account that $b'\in \mathsf C_G(\Lambda)\subset\mathsf C_G(a)$, we conclude that
$ab=ba\in O_B\cap O_A=\emptyset$, which is a desired contradiction proving that 
$B'\subset\varphi(A')$ or $A'\subset\psi(B')$.

Now we can apply Lemma~\ref{l2} and conclude that the partially ordered set $[A]^{<\kappa}$ has $\add([A]^{<\kappa})=\cof([A]^{<\kappa})$, which implies that the ideal $[A]^{<\kappa}$ has a linearly ordered base. Since $|A|=\kappa=|G|$, the ideal $[G]^{<\kappa}$ also has a linearly ordered base. By the (already proved) implication $(1)\Ra(4)$ of Theorem~\ref{t:group2}, the cardinal $\kappa$ is regular. But this contradicts our assumption.
\smallskip

2. Let $H$ be any quotient group of $G$ and $q:G\to H$ be the quotient homomorphism. The definition of the ball structures $\E_{[G]^{<\kappa}}$ and $\E_{[H]^{<\kappa}}$ implies that the quotient homomorphism $q:G\to H$ is open and macro-uniform. Moreover, any section $s:H\to G$ of $q$ is bornologous.  By Proposition~\ref{p:open}, the normality of the ballean $(G,\E_{[G]^{<\kappa}})$ implies the normality of the ballean $(H,\E_{[H]^{<\kappa}})$.
\smallskip

3. Assume that some subgroup $\Gamma$ of $G$ admits a homomorphism on a group $H$ containing a subset $A\subset H$ with $|A|=\kappa=|\mathsf C_H(A)|$.
By Corollary~\ref{c:emb}, the subballean $(\Gamma,\E_{[\Gamma]^{<\kappa}})$ is normal. By the preceding statement, the ballean $(H,\E_{[H]^{<\kappa}})$ is normal. Since $|H|\ge |A|=\kappa=|G|\ge|H|$, the group $H$ has cardinality $\kappa$. By the first statement, $|\mathsf C_H(A)|<\kappa$, which contradicts the choice of the set $A$.
\smallskip

4. To derive a contradiction, assume that the group $G$ is solvable. Let $G^{(0)}=G$ and for every $n\in\w$ let $G^{(n+1)}$ be the commutator subgroup of the group $G^{(n)}$. Since the group $G$ is solvable, for some $n\in\IN$ the group $G^{(n)}$ is trivial. Let $k\in\IN$ be the largest number such that $|G^{(k)}|=\kappa$. Then $G^{(k+1)}$ is a normal subgroup of cardinality $<\kappa$, which implies that the quotient group $H=G^{(k)}/G^{(k+1)}$ is Abelian and has cardinality $|H|=\kappa$. Then the set $A=H$ has $\mathsf C_H(A)=H$ and hence $|A|=\kappa=|\mathsf C_H(A)|$, which contradicts the preceding statement.  
\end{proof}

\section{Some Open Problems}\label{s:op}

In this section we ask some open problems related to normality of products.
Our first problem asks if Theorem~\ref{t:main} can be reversed.

\begin{problem} Let $X,Y$ be two normal balleans of bounded growth whose bornologies have countable base. Is the product  $X\times Y$ normal?
\end{problem}

Our next question concerns separation of cardinal characteristics $\cof_\star(\E_X)$ and $\cof_*(\E_X)$.

\begin{problem} Is $\cof_\star(\E_X)<\cof_*(\E_X)$ for some ballean $(X,\E_X)$?
\end{problem}

By Theorem~\ref{t:Gnorm1}, for every ballean $X$ and every $n\in\IN$ the normality of the $n$-th power $X^n$ implies the normality of the symmetric $n$th power $[X]^n$.

\begin{problem} Let $X$ be a ballean and $n\in\IN$. Is it true that the normality of the symmetric $n$-th power $[X]^n$ implies the normality of the hypersymmetric $n$-th power $[X]^{\le n}$?
\end{problem}

In Theorem~\ref{t:ultranorm} we proved that each ultradiscrete ballean $X$ has normal hypersymmetric powers $[X]^{\le n}$.

\begin{problem} Let $X$ be an ultradiscrete ballean. Is the hyperballean $[X]^{\mathcal B}$ normal?
\end{problem}

\begin{problem}\label{prob:ideal} Assume that for a group ideal $\I$ on a group $G$ the ballean $(G,\E_\I)$ is normal. Has the ideal $\I$ a linearly ordered base?
\end{problem}

By Theorem~\ref{t:group}, Problem~\ref{prob:ideal} has affirmative answer for the group ideal $[G]^{<\kappa}$ of subsets of cardinality $<\kappa$ where $\kappa$ is an infinite regular cardinal. We do not know if Theorem~\ref{t:group} remains true for singular cardinals.

\begin{problem}\label{prob:g} Let $G$ be a group of infinite cardinality $\kappa=|G|$ such that the ballean $(G,\E_{[G]^{<\kappa}})$ is normal. Is the cardinal $\kappa$ regular?
\end{problem}

By Theorem~\ref{t:group2}, for solvable groups the answer to Problem~\ref{prob:g} is affirmative.
\smallskip

We say that a function $\varphi:G\to \IR$ defined on a group $G$ is {\em constant at infinity} if there exists a real number $c$ (denoted by $\lim_{g\to\infty}\varphi(x)$) such that  for any neighborhood $U\subset\IR$ of $c$  the set $G\setminus \varphi^{-1}(U)$ is finite. By Theorem 3.1 of \cite{FP}, every slowly oscillating function $\varphi:G\to\IR$ defined on the finitary ballean $(G,\E_{[G]^{<\w}})$ of an uncountable Abelian group $G$ is constant at infinity. The same result holds more generally for slowly oscillating functions on any uncountable group whose any countable subset is contained in a countable normal subgroup. On the other hand, by Example 3.3 in \cite{FP}, any uncountable free group admits an unbounded slowly oscillating function.

\begin{problem} Let $G$ be a group whose finitary ballean $(G,\E_{[G]^{<\w}})$ is pseudobounded. Is each slowly oscillating function $\varphi:G\to \IR$ constant at infinity?
\end{problem} 

More open problems related to real-valued functions on balleans can be found in \cite{BP}. 
 
 \section{Acknowledgments}
 
 The authors express their thanks to MathOverflow users Nik Weaver, Yves Cornulier and Fedor Petrov for suggesting the idea of the proof of Lemmas~\ref{l:nik1} and \ref{l:nik2} (see \cite{Ban}), and to Yves Cornulier for suggesting a simple proof \cite{MO2} of the inequality $\cof([\kappa]^{<\kappa})>\kappa$ for a singular cardinal $\kappa$,  which allowed us to simplify the proof of the implication $(3)\Ra(4)$ in Theorem~\ref{t:group2}.


\begin{thebibliography}{}

\bibitem{Ban} T.~Banakh, {\em A property of an ultrafilter}, {\tt https://mathoverflow.net/questions/312592}.

\bibitem{MO2} T.~Banakh, {\em 
The cofinality of the poset $[\kappa]^{<\kappa}$ for a singular cardinal $\kappa$}, \newline
  {\tt https://mathoverflow.net/questions/313366}.
  
  \bibitem{BP} T.~Banakh, I.~Protasov, {\em Functional boundedness of balleans}, preprint (https://arxiv.org/abs/1810.12124). 
  


\bibitem{CH} Y.~Cornulier, P.~de la Harpe, {\em Metric geometry of locally compact groups}, EMS Tracts in Mathematics, 25. Z\"urich, 2016. 

\bibitem{DPPZ} D. Dikranjan, I. Protasov, K. Protasova, N. Zava, {\em Balleans, hyperballeans and ideals}, Applied General Topology (accepted).

\bibitem{DZ1} D.~Dikranjan, N.~Zava, {\em Some categorical aspects of coarse spaces and balleans}, Topology Appl. {\bf 225} (2017), 164--194.

\bibitem{DZ2} D.~Dikranjan, N.~Zava, {\em Preservation and reflection of size properties of balleans}, Topology Appl. {\bf 221} (2017), 570--595.

\bibitem{Dr} A.N.~Dranishnikov, {\em Asymptotic topology}, Russian Math. Surveys,  {\bf 55}:6 (2000), 1085--1129.

\bibitem{DH} J.~Dydak, C.~Hoffland, {\em An alternative definition of coarse structures}, Topology Appl. {\bf 155}:9 (2008), 1013--1021.

\bibitem{Eng} R.~Engelking, {\em General Topology}, Heldermann Verlag, Berlin, 1989.

\bibitem{FP} M. Filali, I.V. Protasov, {\em Slowly oscillating functions on locally compact groups}, Applied General Topology, {\bf 6}:1 (2005), 67--77.

\bibitem{LP} I.~Lutsenko, I.~Protasov, {\em Thin subsets of balleans}, Applied General Topology, {\bf 11}:2 (2010), 89--93.

\bibitem{Nek} V. Nekrashevych, {\em Uniformly bounded spaces}, Voprosy Algebry, 14 (1999), 47--67.

\bibitem{NYu} P.~Nowak, G.~Yu, {\em Large scale geometry}, EMS Textbooks in Mathematics, Z\"urich, 2012.

\bibitem{Prot} I.V.~Protasov, {\em Normal ball structures}, Matem. Studii. {\bf 10}:1 (2003), 3--16.

\bibitem{P2008} I.V.~Protasov, {\em Balleans of bounded geometry and $G$-spaces}, Matem. Studii, {\bf 30}:1 (2008), 61--66. 

\bibitem{P18} I.~Protasov, {\em Bornological, coarse and uniform groups}, Visnyk Lviv Univ. (accepted); available at (https://arxiv.org/abs/1807.03028).

\bibitem{PB} I.~Protasov, T.~Banakh, {\em Ball stuctures and colorings of graphs and groups}, VNTL Publ. 2003, 148p. 

\bibitem{ProtEJM} I. Protasov, K. Protasova, {\em On hyperballeans of bounded geometry}, Europ. J. Math. {\bf 4} (2018), 1515--1520.

\bibitem{PZ} I.~Protasov, M.~Zarichnyi, {\em General Asymptology},  VNTL Publ., Lviv, 2007.

\bibitem{Roe} J.~Roe, {\em Lectures on Coarse Geometry}, Univ. Lecture Ser. 31, Amer. Math. Soc., 2003.


\end{thebibliography}
\end{document}